\documentclass[a4paper,11pt]{article}

\usepackage{amsmath,amsthm,amstext,amscd,amssymb,euscript,mathrsfs,url}
\usepackage{calrsfs}
\usepackage{epsfig}

\usepackage{color}
\usepackage{esint}

\usepackage{pdfsync}
\usepackage{graphicx}
\usepackage{todonotes}

\newcommand{\epsi}{\epsilon}
\newcommand{\Z}{\mathbb Z}
\newcommand{\R}{\mathbb R}

\newcommand{\E}{\mathbb E}
\newcommand{\Zd}{\mathbb Z^d}

\newcommand{\eps}{\ensuremath{\epsilon}}

\newcommand{\pee}{\ensuremath{\mathbb{P}}}

\def\1{{\mathchoice {\rm 1\mskip-4mu l} {\rm 1\mskip-4mu l}
{\rm 1\mskip-4.5mu l} {\rm 1\mskip-5mu l}}}

\newtheorem{theorem}{{\small T}{\scriptsize HEOREM}}[section]

\newtheorem{corollary}[theorem]{{\bf{\small C}{\scriptsize OROLLARY}}}
\newtheorem{proposition}[theorem]{{\bf{\small P}{\scriptsize ROPOSITION}}}
\newtheorem{lemma}[theorem]{{\bf{\small L}{\scriptsize EMMA}}}
\newtheorem{remark}[theorem]{{\bf{\small R}{\scriptsize EMARK}}}
\newtheorem{assumption}[theorem]{{\bf{\small C}{\scriptsize ONDITION}}}
\newtheorem{definition}[theorem]{{\bf{\small D}{\scriptsize EFINITION}}}

\renewenvironment{proof}[1]
{\noindent{{\bf{\small{P}{\scriptsize ROOF}}}.}\hspace{0.1cm} #1} {$\;\qed$\newline}

\newcommand{\beq}{\begin{eqnarray}}
\newcommand{\eeq}{\end{eqnarray}}

\newcommand{\ba}{\begin{align*}}
\newcommand{\ea}{\end{align*}}

\newcommand{\be}{\begin{equation}}
\newcommand{\ee}{\end{equation}}

\newcommand{\bl}{\begin{lemma}}
\newcommand{\el}{\end{lemma}}

\newcommand{\br}{\begin{remark}}
\newcommand{\er}{\end{remark}}

\newcommand{\bt}{\begin{theorem}}
\newcommand{\et}{\end{theorem}}

\newcommand{\bd}{\begin{definition}}
\newcommand{\ed}{\end{definition}}

\newcommand{\bp}{\begin{proposition}}
\newcommand{\ep}{\end{proposition}}

\newcommand{\bc}{\begin{corollary}}
\newcommand{\ec}{\end{corollary}}

\newcommand{\bpr}{\begin{proof}}
\newcommand{\epr}{\end{proof}}

\newcommand{\bi}{\begin{itemize}}
\newcommand{\ei}{\end{itemize}}

\newcommand{\ben}{\begin{enumerate}}
\newcommand{\een}{\end{enumerate}}

\newcommand{\SIP}{\text{SIP}}

\newcommand{\SEP}{\text{SEP}}

\renewcommand{\(}{\left(}        \renewcommand{\)}{\right)}
\renewcommand{\[}{\left[}        \renewcommand{\]}{\right]}

     \newcommand{\nn}{\nonumber}

\newcommand{\col}[1]{\textcolor[rgb]{0,0,0}{#1}}

\begin{document}
\title{Exact formulas for two interacting particles and\\
applications in particle systems with duality\\
\vspace{1.cm}}
\author{
Gioia Carinci$^{\textup{{\tiny(a)}}}$\\
Cristian Giardin{\`a}$^{\textup{{\tiny(b)}}}$\\
Frank Redig$^{\textup{{\tiny(a)}}}$
\vspace{1.cm}\\
{\small $^{\textup{(a)}}$ Delft Institute of Applied Mathematics, Technische Universiteit Delft}\\
{\small Mekelweg 4, 2628 CD Delft, The Netherlands}\\
{\small $^{\textup{(b)}}$ Department of Mathematics, University of Modena and Reggio Emilia}\\
{\small via G. Campi 213/b, 41125 Modena, Italy}
}

\maketitle

\begin{abstract}

We consider two particles performing continuous-time nearest neighbor random walk on $\Z$ and interacting with each other when they
are at neighboring positions. Typical examples are two particles in the partial exclusion process  or in the inclusion process.
We provide an {\em exact formula} for the Laplace-Fourier transform of the transition probabilities of the two-particle dynamics.
From this we derive a general {\em scaling limit} result, which shows that the possible scaling limits are coalescing Brownian motions,
reflected Brownian motions, and sticky Brownian motions.

In particle systems with duality, the solution of the dynamics of two dual particles provides relevant information.
We apply the exact formula to the the symmetric inclusion process, that is self-dual, in the  {\em condensation regime}.
We thus obtain two results. First, by computing the time-dependent covariance of the particle occupation 
number at two lattice sites we characterise the time-dependent coarsening
in infinite volume when the process is started from a homogeneous product measure.
Second, we identify the limiting variance of the density field in the diffusive scaling limit,
relating it to the local time of sticky Brownian motion. 


\end{abstract}

\newpage
\section{Introduction}
In interacting particle systems, duality is a powerful tool which enables to study time-dependent
{\em correlation functions of order $n$} with the help of {\em $n$ dual particles}. Examples include 
the exclusion processes \cite{dmp,ligg} and the inclusion processes \cite{gkr}, related diffusion processes 
such as the Brownian Energy process \cite{gkrv}, and  stochastic energy exchange processes, such as 
Kipnis-Marchioro-Presutti model \cite{KMP}.

\vspace{0.1cm}

For systems defined on the infinite lattice $\Zd$, these dual particles typically  spread out and behave on large scale as independent Brownian motions.
This fact is usually enough to prove the hydrodynamic limit in the sense of propagation of local equilibrium \cite{dmp}: under a diffusive scaling limit the
macroscopic density profile evolves according to the heat equation. To study the {\em fluctuations of the density field} one needs to understand 
{\em two dual particles}. Likewise, solving the dynamics of a finite number $n$ of dual particles one gets control on the
$n^{\text{th}}$ moment of the density field.
While the dynamics of one dual particle is usually easy to deal with (being typically a continuous-time random walk),
the  dynamics of $n$ dual particles is usually an hard problem (due to the interaction of the
dual walkers) that can be solved only in special systems, i.e. stochastic integrable systems \cite{borodin0, borodin,schutz,schutz10,tw,spohn}.
In such systems the equations for the time-dependent correlation functions essentially decouple and the transition probability 
of a finite number of dual particles can be solved using methods related to quantum integrability, e.g. Bethe ansatz,Yang Baxter equation, 
and factorized S-matrix. For instance, by studying the two-particle problem for the asymmetric simple
exclusion process, it was proved in \cite{schutz11} that the probability distribution of the particle density of 
only two particles spreads in time diffusively, but with a diffusion coefficient that is notably different from the 
non-interacting case.

\vspace{0.1cm}
In this paper we prove that in dimension $d=1$ and for nearest neighbor jumps with translation invariant rates, a generic system of two (dual) particles turns out to be {\em exactly solvable}, i.e., one can
obtain a closed-form formula for the Laplace-Fourier transform of the transition probabilities in the coordinates of the center of mass and
the distance between the particles. The derivation of this formula, and its applications in particle systems with duality, is the main message of this paper.

\vspace{0.1cm}

The exact formula for the two-particle process will be stated in Theorem \ref{Teo:Green}.
We shall provide conditions on the particles jump rates so that
the formula holds. These conditions include both the case where the particles  have symmetric interactions or the case
where the asymmetry is ``naive'', meaning that it is obtained from the symmetric system by multiplying the rates of jumps to the right
by a parameter $p$ and the rates of jumps to the left by a parameter $q$, with $p\neq q$.

From Theorem \ref{Teo:Green} we obtain a general {\em scaling limit} result (Theorem \ref{scalingthm}), 
which shows that the possible scaling limits of two
interacting random walkers are coalescing Brownian motions (``{absorbed regime'}'),
reflecting Brownian motions (``{reflected regime}'') and {\em sticky Brownian motions} 
(attracting each other via their local intersection time) which interpolate between 
the absorbed and reflected regime, and \col{where the particles spend some positive proportion of time 
at the same place}.

Next, we consider applications of Theorem \ref{Teo:Green} to systems with duality.
We focus, in particular, on the inclusion process \cite{gkr}.
\col{Due to the attractive interaction between particles this model has
a condensation regime \cite{gross,bianchi}. As a consequence of
Theorem \ref{scalingthm}, the scaling limit of the two-particle dynamics
yields sticky Brownian motions.} Furthermore, \col{by using duality}, we study the scaling behavior of the 
variance and covariance of the number of particles in the condensation regime (Theorem \ref{P:variance}). 
This provides better understanding of the coarsening  process (building up of large piles of particles)
when starting from a homogeneous initial product measure in infinite volume.
\col{Last, we study in Theorem \ref{variancethm} the time-dependent variance of the density fluctuation 
field of the inclusion process in the condensation regime. We conjecture the expression 
that we find in Theorem \ref{variancethm}  to be a universal expression describing the time-dependent
variance of the density fluctuation field of systems displaying condensation phenomena, when started 
from a non-equilibrium initial state.}

\section{Model definitions and results}

\subsection{The setting}

We start by defining a general system of particles moving on $\mathbb{Z}$ and interacting when they are nearest neighbor. 
The system is modeled by a continuos-time Markov chain and thus is defined by assigning the process generator.

\begin{definition}[Generator]
\label{gen-def}
Let  $\{\eta(t): \, t\geq 0\}$ be a particle system on the integer lattice, where
$\eta_i(t)\in \Omega\subseteq \mathbb{N}$ 
denotes the number of particles at site ${i\in \mathbb{Z}}$ at time $t$.    
The particles evolve according to the formal generator:
\begin{equation}
\label{general}
 [\mathcal{ L} f](\eta)=\sum_{i\in \mathbb Z}\left\{c_+(\eta_i, \eta_{i+1})\[f(\eta^{i,i+1})-f(\eta)\]+ \,c_-(\eta_{i+1}, \eta_{i})\[f(\eta^{i+1,i})-f(\eta)\]\right\}\,.
\end{equation}
\end{definition}
In the above, $\eta^{i,j}$ is the configuration that is obtained from $\eta\in \Omega^{\mathbb{Z}}$ by removing a particle at site $i$ and adding it at site $j$, i.e.
\begin{equation*}
\eta^{i,j}_k = \left\{
\begin{array}{ll}
\eta_i - 1 & \text{if } k=i,\\
\eta_j + 1 & \text{if } k = j,\\
\eta_k & \text{otherwise } .
\end{array} \right.
\end{equation*}

\noindent
The specific systems that we shall consider below include cases where $\eta_i$, the number of particles at site $i$, has a maximum 
(i.e.  $\Omega$ is a finite set,  as in the partial exclusion processes \cite{schutz}), as well as cases where any natural number is allowed 
(i.e. $\Omega = \mathbb{N}$, as in the inclusion processes \cite{gkr}). 
The rates $c_{+}(\eta_i,\eta_{i+1})$ and $c_{-}(\eta_{i+1},\eta_i)$ in \eqref{general} \col{are allowed to be general} functions of the particle numbers $\eta_i$
and $\eta_{i+1}$ such that there is a well defined Markov process associated to the formal generator $\mathcal{ L}$.

\vskip.2cm
\noindent
When the rates of the right jumps are equal to the rates of the left jumps, i.e.
\begin{equation*}
c_{+}(n,m) = c_{-}(n,m) \qquad \forall \;\; n,m\in \Omega ,
\end{equation*}
we shall say that the system is {\em symmetric}.
For some of our results, symmetry of the jumps is not required.
However, if there is asymmetry, we need it to be of a specific type 
that is explained below.

\subsection{Two interacting particles}

We shall be interested in the case where the process $\{\eta(t): \,t\geq 0\}$ with generator \eqref{general}
is initialized with {\em two particles}.
In the following we will restrict to the case where the following conditions are fulfilled.
\begin{assumption}[Rates of two-particle process]
\label{assumi}
For the rates in \eqref{general}  we assume that:
\begin{itemize}
\item[(i)]
they are translation invariant, i.e. for all
$\eta\in \Omega^{\mathbb{Z}}$ and for all $a\in \mathbb{Z}$
\begin{eqnarray*}
c_{+}(\eta_i,\eta_{i+1}) & =  & c_{+}((\tau_a \eta)_{i+a}, (\tau_a \eta)_{i+1+a})  \\
c_{-}(\eta_{i+1},\eta_i) & = & c_{-}((\tau_a \eta)_{i+1+a}, (\tau_a \eta)_{i+a}), \nonumber
\end{eqnarray*}
where $\tau_a$ denotes the shift by $a$ on $\mathbb{Z}$, i.e.
$(\tau_a \eta)_i = \eta_{i-a}$;
\item[(ii)]  for integers couples $(n,m)$ such
that $n+m \le 2$, they satisfy
\begin{equation}
\label{cond1}
c_+(2,0)+c_-(2,0)=2\(c_+(1,0)+c_-(1,0)\)\;,
\end{equation}
\begin{equation}
\label{cond2}
\frac{c_+(1,0)}{c_-(1,0)} = \frac{c_+(1,1)}{c_-(1,1)} = \frac{c_+(2,0)}{c_-(2,0)}\;.
\end{equation}
\end{itemize}
\end{assumption}

\vspace{0.5cm}
\noindent
Equation \eqref{cond1} is for instance implied by linearity of the rates in the number of particles at the departure site,
i.e. $c_+(\eta_i,0) = A\eta_i$ and $c_{-}(\eta_{i+1},0) = B \eta_{i+1}$ with $A,B > 0$.
Equation \eqref{cond2} refers instead to the degree of asymmetry in the jumps: it is for instance satisfied when
the ratio
$$
\frac{c_+(\eta_i,\eta_{i+1})}{c_-(\eta_i,\eta_{i+1})}
$$
is a constant, i.e. does not depend on $\eta_i,\eta_{i+1}$. This happens in symmetric systems or in systems 
with {\em naive asymmetry} that are obtained from the symmetric systems multiplying by two different constants the rates of the left
jumps and those of the right jumps.

\vspace{0.2cm}

Given rates that fulfill Condition \ref{assumi}, it will be convenient to define three parameters
that  characterize the dynamics of two particles.
\begin{definition}[Parameters] 
We define the parameters
\begin{eqnarray}
&&\alpha:=
c_+(1,0)+c_-(1,0),
\label{alpha-par}\label{alph}\\
&& p := \frac 1 \alpha \, c_+(1,0)
\label{p}
\\
&& \theta := \frac{c_+(1,1)+c_-(1,1)}{c_+(1,0)+c_-(1,0)} - 1
\;.
\end{eqnarray}
\noindent
\end{definition}
\br
Clearly, to define in general the two-particle process one needs to assign six rates, namely $c_{\pm}(1,0)$, $c_{\pm}(2,0)$ and $c_{\pm}(1,1)$.
If Condition \ref{assumi} holds then equations \eqref{cond1} and \eqref{cond2} leave three free choices. The parameters defined above are to be interpreted as follows:
\ben
\item
$\alpha >0$ is total rate for a single particle to jump if both left and right neighboring sites are empty;
\item
$0<p<1$ is the probability of a particle to jump to the right when its neighboring sites are empty;
for convenience we denote by $q=1-p$ the probability of a left jump.
\item
$\theta\in \mathbb{R}$ is a parameter
tuning the strength of the interaction between the
two  particles. In particular, for $\theta=0$ one recovers the case where the particles perform two independent random walks.
\een
\er

\subsection{Distance and center of mass coordinates}

One main goal of this paper is to achieve a full control of the dynamics of the two-particle process.
To this aim it will be convenient to move to new coordinates.
Consider the process $\{\eta(t):t\geq 0\}$ with generator \eqref{general} initialized 
with two particles and denote by $(x_1(t),x_2(t))$ the particle positions at time $t$, 
with an arbitrary labeling of the particles,
but fixed once for all. Define the {\em distance} and {\em sum} coordinates by
\begin{equation}\label{wup}
\begin{array}{ll}
w(t):=|x_2(t)-x_1(t)|\;,\\
u(t):=\;\; x_1(t)+x_2(t)\;.
\end{array}
\end{equation}
By definition, the  distance and sum coordinates are not depending on the chosen labeling of particles.
As a consequence of the fact that the size of the particle jumps is one, both the difference and the sum coordinates change by one unit at every particle jump. Therefore they both perform continuous-time 
simple random walk.
Under  Condition \ref{assumi},  a \col{straightforward} computation starting from  \eqref{general}
shows that the distance process $\{w(t):\, t\ge 0\}$, that is valued in $\mathbb{N} \cup \{0\}$, evolves 
according to the generator
\begin{equation}
\label{gen-w}
[{\mathcal L}^{(dist)} f](w)= \left\{
\begin{array}{ll}
2 \alpha (f(w+1) -f(w)) & \text{if } w=0,\\
& \\
\alpha  (f(w+1) -f(w)) + \alpha (\theta+1) (f(w-1) -f(w)) & \text{if } w=1,\\
& \\
\alpha  (f(w+1) -2f(w) + f(w-1)) & \text{if } w \ge 2 .
\end{array} \right.
\end{equation}
As for the sum coordinate $\{u(t)\,:\; t \ge 0\}$, this is a process that is valued in $\mathbb{Z}$.
One finds that, {\em conditionally on $\{w(t), \, t\ge 0\}$}, it evolves according to
\begin{eqnarray}\label{ucond}
&&[{\mathcal L}^{(sum)} f](u)=2\alpha\;\(1+\frac \theta 2\,\mathbf 1_{w=1}\)\big\{p[f(u+1)-f(u)] + q [f(u-1)-f(u)]\big\}\;.
\end{eqnarray}


We thus see that the distance between the particles evolves in an {\em autonomous} way 
as a symmetric random walk on the integers, reflected at 0 and with a defect in 1. The sum coordinate instead
is dependent on the distance through its jump rate. More precisely, the sum performs
an asymmetric continuous-time nearest neighbor random walk driven by a Poisson process whose rate
depends on the distance process. These properties of the distance and sum coordinates
will be the key properties that will be used in the exact solution of the two-particle dynamics.

Such exact solution will be expressed by considering the Fourier-Laplace transform
of the transition probability
\be
\label{trans}
P_t((u,w),(u',w')) = \mathbb{P}(u(t) = u', w(t) = w' \,|\, u(0) = u, w(0) = w),
\ee
where $\mathbb{P}$ denotes the law of the two-particle process.
As it can be seen from the generators \eqref{gen-w} and \eqref{ucond},
these transition probabilities are translation invariant only in the sum coordinate,
i.e., $\col{P}_t((u,w),(u',w'))= \col{P}_t((0,w), (u'-u,w'))$, and therefore it is natural that we 
take Fourier transform w.r.t. the sum coordinate. Furthermore it will also be convenient to 
take Laplace transform w.r.t. time.
\bd[Fourier-Laplace transform of the transition probability]
\label{main-def}
Let the parameter $\alpha$ defined in \eqref{alpha-par} be equal to $1$ and
let $\col{P}_t\((u,w),(u',w')\)$ be the transition probability \col{in} \eqref{trans}.
We define the Laplace transform of the transition probability 
\begin{equation}\label{GL}
\mathcal{G}^{(\theta)}((u,w),(u',w');\lambda):= \int_0^\infty e^{-\lambda t }\col{P}_t\((u,w),(u',w')\)\, dt, \qquad \lambda\ge 0
\end{equation}
and its Fourier transform 
\begin{equation}\label{GF}
G^{(\theta)}(w,w',\kappa,\lambda):=\sum_{v \in \mathbb Z}e^{-i\kappa v}\, \mathcal{G}^{(\theta)}\((0,w),(v,w');\lambda\), \qquad \kappa \in \mathbb{R}\;.
\end{equation}
\ed
\br[Changing $\alpha$] \label{alfa}
Notice that the parameter $\alpha>0$ in \eqref{ucond} appears as a multiplicative factor in the generator, therefore
for a generic value of this parameter we have the scaling property
\be
\label{scale}
\mathcal G^{(\theta,\alpha)}((u,w),(u',w');\lambda)= \frac{1}{\alpha}\mathcal{G}^{(\theta)}\((u,w),(u',w');\frac{\lambda}{\alpha}\),
\ee
where we made the $\alpha$-dependence explicit.
The $\mathcal{G}^{(\theta)}$  in \eqref{GL} coincides with $\mathcal G^{(\theta,1)}$ and for a generic value of $\alpha$ we can use \eqref{scale}.
\er

\subsection{Processes with duality}
\label{dual-section}

Besides the inherent interest of studying the two particle process (and its scaling limits) we will be interested  
in applications of the exact solution of the two-particle dynamics. This applications will be given 
in the context of interacting particle systems with self-duality.  
\begin{definition}[Self-duality]
Let $\{\eta(t):t\geq 0\}$ be a process of type introduced in Definition \ref{gen-def}. 
We say that the process is self-dual with self-duality function 
$D: \Omega \times \Omega \to \mathbb{R}$ if
for all $t\geq 0$ and for all $\eta,\xi \in \Omega^{\mathbb{Z}}$ 
we have the self-duality relation
\be\label{selfdualityrel}
\E_\eta D(\xi,\eta(t))= \E_\xi D(\xi(t),\eta),
\ee
where $\{\xi(t):t\geq 0\}$ is an independent copy of the process with generator
\eqref{general}. In the above $\E_\eta$ on the l.h.s. denotes expectation in
the original process initialized from the configuration $\eta$ and 
$\E_\xi$ on the r.h.s. denotes expectation in
the copy process initialized from the configuration $\xi$.
We shall call $D$ a factorized self-duality function when
\be
D(\xi,\eta)= \prod_{i\in \mathbb{Z}} d(\xi_i,\eta_i).
\ee
The function $d(\cdot,\cdot)$ is then called the single-site 
self-duality function.
\end{definition}
\noindent
\col{For self-dual processes the dynamics of two particles provides relevant information about
the time-dependent correlation functions of degree two}.
The simplest possible choice satisfying  Conditions \ref{assumi} and the 
self-duality property is the choice of rates that are linear in both the departure 
and arrival sites. For later convenience we  shall
refer to the process with this choice of the 
rates as the {\em reference process}. It is given by the generator
\begin{equation}
\label{referencep}
\mathcal [\mathcal{ L}_{\text{ref}} f](\eta)=	\col{\frac{\alpha}{2}}\sum_{i\in \mathbb Z}\left\{\eta_i(1+\theta \eta_{i+1})\[f(\eta^{i,i+1})-f(\eta)\]+\eta_{i+1}(1+\theta \eta_i) \,\[f(\eta^{i+1,i})-f(\eta)\]\right\}\,.
\end{equation}
This is a process whose behavior depends on the sign of $\theta$.
It corresponds to the exclusion process for $\theta<0$, to the inclusion process for $\theta>0$, and
to independent random walkers for $\theta=0$. 
We now explain the precise connection between these processes 
and the reference process \eqref{referencep}.

\subsubsection{The symmetric inclusion process}
The Symmetric Inclusion Process with parameter $k>0$, denoted $\SIP(k)$, is the process with generator \col{\cite{gkr}}
\be
\label{sip-gen}
L_{\SIP(k)} f(\eta)=\sum_{i\in\Z}\left( \eta_i(k+\eta_{i+1})\nabla_{i,i+1} + \eta_{i+1} (k +\eta_i)\nabla_{i+1,i}\right)f(\eta),
\ee
where $\nabla_{i,\ell} f(\eta)= f(\eta^{i,\ell})-f(\eta)$.
This amounts to choose the parameters in the reference process \eqref{referencep} as follows
\beq\label{sippara}
\theta=\frac{1}{k}\qquad\text{and}\qquad \alpha=\frac{2}{\theta} =2k,
\eeq
Conversely, the reference process with generator \eqref{referencep} corresponds to 
a time rescaling of the $\SIP$ process, i.e.
\be
\label{dictionary}
\eta^{\text{ref}}(t) = \eta^{\SIP(\frac{1}{\theta})}\left(\frac{\theta \alpha}{2}t\right).
\ee
The $\SIP(k)$ is self dual with single-site self-duality function:
\beq\label{dsip}
d_{\SIP(k)}(m,n)=
\frac{n!\Gamma(k)}{(n-m)! \Gamma(k + m)}\;\mathbf 1_{\{ m\leq n\}}.
\eeq
This self-duality property with self-duality function \eqref{dsip} continues to hold when
$\theta=\frac{1}{k}$ for all
other values of $\alpha>0$.
\subsubsection{The  symmetric partial exclusion process}
 \col{We recall the definition of} the Symmetric partial Exclusion Process with parameter $j\in \mathbb N$, $\SEP(j)$ \col{\cite{SS}}. Notice that $j$,
 that  is the maximum number of particles allowed for each site,
 has to be a natural number. For $j=1$ the process is the standard Symmetric Exclusion Process.
The generator is
\be
\label{sep-gen}
L_{\SEP(j)} f(\eta)= \sum_{i\in\Z} \left( \eta_i(j-\eta_{i+1})\nabla_{i,i+1} + \eta_{i+1} (j -\eta_i)\nabla_{i+1,i}\right)f(\eta),
\ee
i.e., comparing with the reference process \eqref{referencep} we have
\beq\label{sippara2}
\theta=-\frac{1}{j},\qquad\alpha=-\frac{2}{\theta}=2j.
\eeq
The Symmetric partial Exclusion Process $\SEP(j)$ is self-dual with  single-site self-duality function:
\beq\label{dsep}
d_{\SEP(j)}(m,n) =
\frac{{n\choose m}}{{j\choose m}}\;\mathbf 1_{\{ m\leq n\}}.
\eeq
As before, this self-duality property with self-duality function \eqref{dsep} continues to hold when $\theta=-1/j$ for all
other values of $\alpha>0$.

\subsubsection{Independent symmetric random walk}
 The last example is provided by a system of independent random walkers (IRW). In this case
the generator is
\be
\label{irw-gen}
L_{\text{IRW}} f(\eta)= \sum_{i\in\Z} \left(\eta_i\nabla_{i,i+1} +\eta_{i+1}\nabla_{i+1,i}\right)f(\eta),
\ee
which implies, comparing with the reference process \eqref{referencep}, that
\beq\label{nasipparaaa}
\alpha
=2, \qquad  \theta=0.
\eeq
In this process we have self-duality with  single-site self-duality function:
\beq\label{dirw}
\col{d_{\text{IRW}}(m,n)={\frac{n!}{(n-m)!}}\;\mathbf 1_{\{ m\leq n\}}.}
\eeq
As before, this self-duality property with self-duality function \eqref{dirw} continues to hold when $\theta=0$ for all
other values of $\alpha>0$.

\subsection{Main results}
\label{main-results}

We state  now  our main results. Without loss of generality,
as it was done in Definition \ref{main-def},  we will always choose in the following 
$\alpha =1$, where $\alpha$ is the parameter
defined in \eqref{alph}. 
The case of general $\alpha$ just corresponds to a rescaling of time, 
i.e. $t'=\alpha t$ (cf. Remark \ref{alfa}).

\subsubsection{Exact solution of the two-particle dynamics}

We start by providing the formula for the \col{Fourier-Laplace} transform
of the transition probability of the distance and sum coordinates.

\bt[Fourier-Laplace transform for the distance and sum coordinates]
\label{Teo:Green}
Under  Condition \ref{assumi},
the Fourier-Laplace transform in Definition \ref{main-def} is given by
\begin{eqnarray}
 \label{G-main}
&& G^{(\theta)}(w,w',\kappa,\lambda)= \frac {f^{(\theta)}_{\lambda,\kappa}(w,w')} {\mathcal Z^{(0)}_{\lambda,\kappa}}
\left\{ \zeta_{\lambda,\kappa}^{|w'-w|-1}+\zeta_{\lambda,\kappa}^{w'+w-1} 
\left( 2\,  \frac{\mathcal Z_{\lambda,\kappa}^{(0)} }{\mathcal Z^{(\theta)}_{\lambda,\kappa}}-1 \right)\right\},
\end{eqnarray}
with
%
%
\begin{equation}
\label{f-funct}
f_{\lambda,\kappa}^{(\theta)}(w,w')= \left\{
\begin{array}{ll}
\frac{\theta \nu_{\kappa}^{-1}\zeta_{\lambda,\kappa} +1}{2} & \text{if  \; } w=0, w'=0\\
& \\
\frac{\theta+1}{2}& \text{if \; } w\ge1, w'=0\\
& \\
1 & \text{if \; } w \ge 1, w' \ge 1,
\end{array} \right.
\end{equation}
and
\begin{eqnarray}
\mathcal Z^{(\theta)}_{\lambda,\kappa}
&=&\nu_\kappa (\zeta_{\lambda,\kappa}^{-2}-1) +2\theta (x_{\lambda,\kappa}-\nu_\kappa),
\end{eqnarray}
where
\begin{equation}\label{zeta}
 \zeta_{\lambda,\kappa}:=\zeta\(x_{\lambda,\kappa}\)=x_{\lambda,\kappa}-\sqrt{x_{\lambda,\kappa}^2-1}, \qquad x_{\lambda,\kappa}:=\frac 1 {\nu_\kappa}\(1+\frac \lambda {2}\),
\end{equation}
\be
\nonumber
\nu_\kappa=\cos(\kappa)-i (p-q) \, \sin(\kappa).
\ee
\\
\et
\br[Meaning of $\nu_k$ and $\zeta(x)$]
One recognizes that  $\nu_\kappa$ is the Fourier transform of the increments of the discrete time asymmetric random walk on $\mathbb{Z}$ 
moving with probability $p$ to the right and $q$ to the left.
Furthermore, as it will be clear from the proof of Theorem \ref{Teo:Green}, the function $ \zeta(x)$ appearing in \eqref{zeta} is related to the probability generating function of $S_0$, the first hitting time of the origin $0$ of the discrete time asymmetric random walk moving with probability $p$ to the right and $q$ to the left starting at 1 at time zero. More precisely
\be
\zeta(x) = \mathbb{E}_1(x^{-S_0}), \qquad \qquad x\ge 1\;.
\ee
\er

In order to give more intuition for the formula in Theorem \ref{Teo:Green}, we transform to leftmost and rightmost position coordinates.
In this coordinates the comparison between the interacting ($\theta\neq 0)$ and non-interacting ($\theta = 0)$ case
becomes more transparent.
Let $(x(t),y(t))$ be the coordinates defined by
\be
x(t) := \min\{x_1(t),x_2(t)\} \qquad y(t) := \max\{x_1(t),x_2(t)\},
\ee
where  $(x_1(t),x_2(t))$ denote the particle positions. We define the Laplace transform
\begin{equation}
\label{laplace-left-right}
\Pi^{(\theta)}((x,y),(x',y');\lambda):= \int_0^\infty \pi_t\((x,y),(x',y')\) e^{-\lambda t} dt,
\end{equation}
where $\pi_t$ denotes the transition probability of the process
$\{ (x(t),y(t)): t \ge 0\}$ started from $(x,y) \in \mathbb{Z} \times  \mathbb{Z}$.

\begin{corollary}[Positions of leftmost and rightmost particle]
\label{for}
Under  Condition \ref{assumi} and assuming  $x\neq y$,
the Laplace transform in \eqref{laplace-left-right} is given by
 \begin{equation}
 \label{left-right-formula}
 \Pi^{(\theta)}((x,y),(x',y');\lambda)=
\left\{
\begin{array}{ll}
A^{(\theta)}_{+}(x'-x,y'-y,\lambda)+A^{(\theta)}_{-}(y'-x,x'-y,\lambda)& \quad \text{if} \quad y'>x'\\
A^{(\theta)}_{0}(x'-x,x'-y,\lambda) & \quad\text{if} \quad y'=x'
\end{array}
\right.
\end{equation}
where
\be
A_{\pm,0}^{(\theta)}(x,y,\lambda):=
\frac {1}{8\pi^2}\int_{-\pi}^\pi \, \int_{-\pi}^\pi \, \frac{\Gamma^{(\theta)}_{\pm,0}(\frac{\kappa_1+\kappa_2}2,\lambda)\,e^{ i(\kappa_1x+\kappa_2y)}}{
1 + \frac{\lambda}{2}- \(\cos(\frac{\kappa_2+\kappa_1}2)-i(p-q)\sin\(\frac{\kappa_2+\kappa_1}2\)\)\cos(\frac{\kappa_2-\kappa_1}2)} \, d\kappa_1\, d\kappa_2
\ee
and
\begin{eqnarray}
 \Gamma^{(\theta)}_{+}(\kappa,\lambda):=1, \qquad \quad \Gamma^{(\theta)}_{-}(\kappa,\lambda):= 2\,  \frac{\mathcal Z_{\lambda,\kappa}^{(0)} }{\mathcal Z^{(\theta)}_{\lambda,\kappa}}-1, \qquad \quad
 \Gamma^{(\theta)}_{0}(\kappa,\lambda) :=  \(\theta+1 \)\,  \frac{\mathcal Z_{\lambda,\kappa}^{(0)} }{\mathcal Z^{(\theta)}_{\lambda,\kappa}}.
\end{eqnarray}
\end{corollary}

From the above formula we immediately see that the Fourier transform of $A^{(\theta)}_{\pm,0}(x,y,\lambda)$ is given by
\begin{eqnarray}
\label{acca}
\widehat A_{\pm,0}^{(\theta)}(\kappa_1,\kappa_2,\lambda)=
\frac{\Gamma^{(\theta)}_{\pm,0}(\frac{\kappa_1+\kappa_2}2,\lambda)}{
2 + \lambda-(\cos\kappa_1+\cos\kappa_2)+(p-q)i(\sin\kappa_1+\sin\kappa_2)}.
\end{eqnarray}
Notice
that for $\theta=0$ we have $\Gamma^{(0)}_{\pm,0}(\kappa,\lambda)=1$, and thus  
we recover the  Fourier-Laplace transform of the transition probability of two independent random walkers. 

\subsubsection{Scaling limits in the symmetric case}

Our second main result is related to the characterization of the scaling limit of the two-particle process in the symmetric case
(\col{corresponding to the parameter $p$ in \eqref{p} equal to 1/2}). We thus consider a diffusive scaling of space and time. 
Being $\alpha=1$ (cf. the beginning of Section \ref{main-results}), this leaves only $\theta >0$ as a free parameter.
Given a scaling parameter $\epsilon>0$, we  
define
\be
U_\epsilon(t):=\frac{\eps\, u(\eps^{-2}t)}{\sqrt 2}, \qquad\quad W_\epsilon(t):=\frac{\eps\, w(\eps^{-2}t)}{\sqrt 2} \;.
\ee
We also assume that the \col{initial values can depend on $\epsilon$, i.e. $u_{\eps}= u(0)$ and $w_{\epsilon} = w(0)$, and we define}
\be\label{i.c.}
U:=\lim_{\eps\to 0}\frac{\eps \col{u_{\eps}}}{\sqrt 2},\qquad\quad W:=\lim_{\eps\to 0}\frac{\eps \col{w_{\eps}}}{\sqrt 2}.
\ee
\col{with $U\in \mathbb{R}$ and $W\in\mathbb{R_+}$}.
\col{Similarly} we suppose $\theta$ to be a function of $\epsilon$, and thus write $\theta_\eps$ and we distinguish three different regimes as $\eps \to 0$:
\begin{itemize}
\item[(a)] {\bf Reflected Regime:}
$\lim_{\eps\to 0}\eps\theta_\eps =0$
\item[(b)]  {\bf Sticky Regime:} $\theta_\eps>0$ and $\eps\theta_\eps= O(1)$. In this regime we define
\be\label{gamma}
 \gamma:= \lim_{\eps\to0}\frac{\eps\theta_\eps}{\sqrt 2 } \in (0,\infty)
\ee
\item[(c)] {\bf Absorbed Regime:} $\theta_\eps>0$ and
$
 \lim_{\eps\to0}\eps\theta_\eps= +\infty
$
\end{itemize}
We have the following result.

\bt[Scaling limits]\label{scalingthm}
Let $\{B(t) : t \ge 0\}$ and $\{\tilde B(t) : t \ge 0\}$ be two independent Brownian motions starting, respectively,  at $W\ge 0$ and at the origin $0$.
Let $s(t)$ be defined by
\be
s^{-1}(t) = t + \gamma L(t),
\ee
where $L(t)$ is the local time at the origin of $B(t)$,
so that $\{B^S(t) = |B(s(t))| : t\ge 0\}$ is the one-sided \col{sticky} Brownian  motion started at $W$,
\col{with stickiness  at the origin of  parameter $\gamma \in [0,\infty]$}. 
Let $\{B^R(t) : t \ge 0\}$ denote the Brownian motion 
reflected at the origin started at $W\ge 0$ and let \col{$\{B^A(t): t \ge 0 \}$} 
denote the Brownian motion absorbed at the origin started at $W\ge 0$. 
Then the following holds true: under Condition \ref{assumi} we have
\begin{equation}\label{thetaepso}
\lim_{\eps\to 0} \left((U_\eps(t)-U), W_\eps(t)\right)=(U(t),W(t))
\end{equation}
where \col{$\{(U(t),W(t)) : t \ge 0\}$ is defined by $(U(0),W(0)) = (0,W)$ and }
\begin{equation}
\label{pluto}
(U(t),W(t))=
\left\{
\begin{array}{ll}
\col{(\tilde B(t),B^R(t))}
& \quad\text{in the Reflected Regime}\\
\\
\col{(\tilde B(2t-s(t)),B^S(t))} 
& \quad\text{in the Sticky Regime} \\
\\
\col{(\tilde B(2t-t\wedge \tau_W), B^A(t))} 
& \quad\text{in the Absorbed Regime }
\end{array}
\right.
\end{equation}
where the convergence in \eqref{thetaepso} is in the sense of finite-dimensional distributions and
\col{$\tau_{W}$ in the third line of \eqref{pluto} is the absorption time of $\{B^A(t) : t \ge 0\}$}.
\et

Thus the scaling limit of the two particle process turns out to be two Brownian motions with ``sticky interaction'',
\col{that can be thought of as an interpolation between  two coalescing Brownian motions and two reflecting Brownian motions}.
More precisely, the distance between the particles converges
to a sticky Brownian motion, which in turn has two limiting cases, namely the absorbed and 
reflected Brownian motion.  
On the other hand, the sum of the particle positions becomes a process which is subjected to the
sticky Brownian motion driving the difference and is ``moving at faster rate'' when the particles are together, 
i.e., it is a time-changes Brownian motion of which the clock runs faster with an acceleration determined by 
the local intersection time.

\br[The symmetric inclusion process in the condensation regime]
For the symmetric inclusion process $\SIP(k)$ we say that we are in the condensation regime when the parameter
$k$ tends to zero sufficiently fast,
i.e., when \col{the spreading of the particles} is much slower than the attractive interaction due to the inclusion 
jumps \col{\cite{gross,bianchi}}.
After a suitable rescaling two $\SIP(k)$ particles will then behave as independent Brownian motions which spend
``excessive'' local time together.
The {\em sticky regime} with stickiness parameter $\gamma\in (0,\infty)$ corresponds to the choice
$k= \tfrac{\epsi}{\gamma\sqrt{2}}$, and acceleration of time by a factor $\epsi^{-3}\frac{\gamma}{\sqrt{2}}$.
I.e., this corresponds to the condensation regime $k\to 0$, where time is diffusively rescaled, and
speeded up with an extra $\epsi^{-1}$ in order to compensate for the vanishing diffusion rate.
\er

\br[Exclusion particles scale to reflected Brownian motions]
\label{exclusion-scaling}
For  the exclusion process SEP$(j)$ it is not possible to consider the sticky or the absorbed regime, because $\theta<0$. For this reason we only scale  time diffusively with a factor $2j\eps^{-2}$ and  take $\theta=-1/j$ fixed.
This then corresponds to consider the reflected regime in \eqref{thetaepso} where $(U_\eps(t)-U,W_\eps(t))$  converge to
\col{$(\tilde B(t), |B(t)|)$ where $\tilde B(t)$ is a standard Brownian motion and $B(t)$ is an independent  Brownian motion started at $W$}.
\er

Next we show that the \col{expected} local time of the difference process $\{W_\eps(t),t\geq 0\}$ converges
to the \col{expected} local time of the limiting sticky Brownian motion \col{(in the sense of convergence of the Laplace transform)}. 
Notice that this does not follow from weak 
convergence of the previous Theorem \ref{scalingthm}, but has to
be viewed rather as a result in the spirit of \col{a local limit theorem}.
\bp[Local time in 0]
\label{local0}
We have
\be
\label{primus}
\int_0^\infty e^{-\lambda t}\,\col{\mathbb P_{w}\(w(t)=0\)}\, dt=
 { \zeta_{\lambda}^{w}  }\; \frac{1+\theta \zeta_{\lambda}^{\mathbf 1_{w=0}} }{\zeta_{\lambda}^{-1}+(\theta \lambda-1)\zeta_{\lambda}}
\ee
with
\begin{eqnarray}\label{Zeta}
  \zeta_\lambda:=\zeta_{\lambda,0}:=1+\frac \lambda 2-\sqrt{\lambda+\frac{\lambda^2}4}
\end{eqnarray}
As a consequence, in the Sticky regime we have
\be\label{localstick}
\lim_{\eps \to 0}\int_0^\infty e^{-\lambda t}\,\col{\mathbb P_{w_{\epsilon}}\(W_\eps(t)=0\)}\, dt=
 \frac{ \gamma }{\sqrt{2\lambda}+\gamma \lambda}\,e^{-\sqrt{2\lambda}W}\;
\ee
with $W$ as in \eqref{i.c.}, $\gamma$ as in \eqref{gamma}.
\ep
\br
Notice that the r.h.s. of \eqref{localstick}
is exactly the Laplace transform of the probability  $\mathbb{P}_W(B^S(t) =0)$ of the sticky Brownian motion started at $W\ge 0$ to be at the origin at time $t$, see Lemma \ref{L:ciao}.
\er

\subsubsection{Coarsening in the condensation regime of the inclusion process}

\col{
We now present some results for the symmetric inclusion process $\SIP(k)$,
which is a self-dual process.
Let the time-dependent covariances of the particle numbers at sites $x\in\mathbb{Z}$ and $y\in\mathbb{Z}$ at time $t\ge 0$ be defined as
\be
\label{xi-def}
\Xi^{(\theta)}(t,x,y;\nu)= \int \E_{\eta}\left[ (\eta_x(t)-\rho_x(t))(\eta_y(t)-\rho_y(t))\right] d\nu(\eta),
\ee
where $\nu$ denotes the initial measure (i.e. the initial distribution of the particle numbers) and 
\be\rho_x(t) =  \int \E_{\eta}\left[ \eta_x(t)\right] d\nu(\eta).
\ee
}
The following Theorem gives an explicit result for the variance and the covariance 
of the time-dependent particle numbers
\col{in the {\em sticky regime} of the symmetric inclusion process}  when 
starting from a homogeneous product measure \col{in infinite volume}. In particular, 
we see how the variance diverges when \col{the inclusion parameter $k=1/\theta$ 
goes to zero, which corresponds to the condensation limit} with piling up of particles.

\bt[\col{Scaling of variance and covariances in the sticky regime of the inclusion process}]\label{P:variance}
Let $\{\eta(t) : t \ge 0\}$ be the reference process  with generator \eqref{referencep} with $\alpha=1$ and $\theta>0$
(i.e. the time rescaled inclusion process, see \eqref{dictionary}).
Suppose we are in the Sticky Regime, i.e.  $\eps\theta_\eps= \col{O(1)}$ as $\eps \to 0$.
Let $\gamma$ as in \eqref{gamma} and  let  $a>0$.  
Then, for any initial homogeneous  product measure  $\nu$,  we have:
\begin{itemize}
\item[a)] {\em Scaling of covariances}. If $x\neq y$, then  
\begin{equation}\label{covariance}
\int_0^\infty e^{-\lambda t}\, \Xi^{(\theta_\eps)}(\eps^{-a}t,\lfloor x\eps^{-1}\rfloor,\lfloor y\eps^{-1}\rfloor;\nu) \,dt=
\left\{
\begin{array}{ll}
-\frac{\eps^{\tfrac a  2 -1}}{\sqrt 2 \gamma \lambda}\, e^{-\sqrt{\lambda}|x-y|\eps^{\tfrac a 2-1}}(1+o(1)) & \text{for }1<a<2,\\
\\
-\, \frac{\gamma \rho^2\,e^{-\sqrt{\lambda}|x-y|}}{\sqrt{2\lambda}+\gamma \lambda}\,(1+o(1)) & \text{for }a=2,\\
\\
-\frac{\gamma\rho^2}{\sqrt {2 \lambda}}\;\eps^{\tfrac{a}{2} -1}(1+o(1)) & \text{for } a>2.
\end{array}
\right.
\end{equation}
\item[b)] {\em Scaling of variance.}
If $x=y$, then
\end{itemize}
\be\label{scvariance}
\int_0^\infty e^{-\lambda t}\, \Xi^{(\theta_\eps)}(\eps^{-a}t,\lfloor x\eps^{-1}\rfloor ,\lfloor x\eps^{-1}\rfloor ; \nu) \,dt=
\left\{
\begin{array}{ll}
\frac{2 \rho^2}{\lambda\sqrt \lambda}\,\eps^{-\tfrac a 2}(1+o(1))& \text{for }1<a<2,\\
\\
\, \frac{2 \sqrt 2 \gamma\rho^2}{2\lambda+\gamma \lambda\sqrt{2\lambda}}\,\eps^{-1}\,(1+o(1)) & \text{for }a=2,\\
\\
\frac{\sqrt 2 \, \gamma \rho^2}\lambda \, \eps^{-1}(1+o(1)) &\text{for }a>2.
\end{array}
\right.
\ee
\et

\br[Coarseing]
We see that the r.h.s. of 
\eqref{covariance} has three different regimes which can intuitively be understood as follows.
\ben
\item{\em Subcritical time scale.} In the
first regime, corresponding to ``short times'' we see that the covariance goes to zero as $\epsi\to 0$, which
is a consequence of the initial product measure structure. At the same time we see a scaling corresponding to the Laplace transform of expected local intersection time of coalescing Brownian motions
\col{(cf.\  limit of $\gamma\to\infty$ of \eqref{localstick}.)} This corresponds to the dynamics of large piles (at typical distance $\eps^{-1}$) which merge as coalescing Brownian motions, because
on the time scale under consideration there is no possibility to detach. 
 \item{\em Critical time scale.} In the second regime \col{corresponding to} ``intermediate times'' we see a scaling corresponding to the Laplace transform of expected local intersection time of sticky Brownian motions.
This identifies the correct scale at which the ``piles'' have a non-trivial dynamics, i.e., can interact, merge and detach. This is also the correct time scale for the density fluctuation field (cf.\ Theorem \ref{variancethm}).
\item {\em Supercritical time scale.} 
In the last regime, the covariance is $o(1)$ as $\epsi\to 0$, which corresponds to the stationary regime, in which again a product measure is appearing.
The $1/\sqrt{\lambda}$ scaling corresponds in time variable to $1/\sqrt t$, which corresponds to the probability density of two independent Brownian motions, initially
at $\epsi^{-1}x, \epsi^{-1}y$ to meet after a time $\epsi^{-a}t$, indeed: 
\begin{equation*}
\frac{\exp\left\{ -\frac{ (x-y)^2\epsi^{-2}}{2t\epsi^{-a}}\right\}}{\sqrt{2\pi t \epsi^{-a}}}\approx \frac{\epsi^{a/2}}{\sqrt{2\pi t}} 
\end{equation*}
This corresponds to the fact that on that longer time scale, the stickyness of the piles disappears  and they move as independent particles, unless they are together.
\een
\er

\subsubsection{Variance of the density field in the condensation regime of the inclusion process}

\col{Having identified the relevant scaling of the variance and covariance of the time-dependent particle number, 
we apply this to compute the limiting variance of the rescaled density fluctuation field, which shows a non-trivial 
limiting dependence structure in space and time.}
We consider the density fluctuation field out of equilibrium, i.e., we start the process from an homogeneous invariant 
product measure $\nu$ which is not the stationary distribution and has expected particle number $\int \eta_x d\nu= \rho$ 
\col{for all $x\in\mathbb{Z}$}. More precisely, we study the behavior of  the random time-dependent distribution
which is defined by its action on a Schwarz function \col{$\Phi: \mathbb{R}\to \mathbb{R}$} via
\be\label{densfi}
\mathcal X_\eps(\Phi,\eta,t)= \eps \, \sum_{x\in\mathbb{Z}}\Phi(\eps x)( \eta_x({\eps^{-2}t}) -\rho)
\ee
where $\eta$ in the l.h.s. of \eqref{densfi} refers to the intitial configuration $\eta(0)$ which is distributed according to $\nu$.
Notice that we multiply by $\epsi$ in \eqref{densfi} as opposed to the more common $\sqrt{\epsi}$ which typically  appears in fluctuation fields of particle systems (in dimension one) with a conserved quantity and which then usually converges  to an infinite dimensional Ornstein-Uehlenbeck process see e.g. \cite{landim}, Chapter 11. Here, on the contrary, we are in the condensation regime, and therefore the variance of the particle occupation numbers is of order $\epsi^{-1}$ by Theorem \ref{P:variance}, which explains why we have to multiply with an additional factor
$\sqrt{\epsi}$ in comparison with the standard setting.

\bt[Variance of the density fluctuation field]
\label{variancethm}
Let $\{\eta(t) : t \ge 0\}$ be the reference process  with generator \eqref{referencep} with $\alpha=1$ and $\theta>0$
(i.e. the time rescaled inclusion process, see \eqref{dictionary}).
Assume we are in the sticky regime, i.e.  $\theta=\eps\theta_\eps=\mathcal O(1)$ as $\eps \to 0$ and let $\gamma$ as in \eqref{gamma}. 
Let $\nu$  be an initial homogeneous product measure then
\beq\label{vardenslim}
&&\lim_{\eps\to 0}\int_0^\infty e^{-\lambda t}\;\E_{\nu}\big[ \left(\mathcal X_\eps(\Phi,\eta,t)\right)^2\big]\, dt=\nn\\
&& \hskip1cm=\frac{\gamma \rho^2\,}{\sqrt{2\lambda}+\gamma \lambda}\int \Phi(x)\Phi(y)\, e^{-\sqrt{\lambda} |x-y|}\, dx\, dy+ \frac{2 \sqrt 2\, \gamma\rho^2}{2\lambda+\gamma \lambda\sqrt{2\lambda}}
 \int \Phi(x)^2\, dx.
\eeq
\et
\noindent
The limiting variance of the density fluctuation field consists of two terms which both contain the stickyness parameter $\gamma$. The combination of both terms describe how from the initial homogeneous  measure $\nu$ one enters the condensation regime.
Comparing to the standard case of e.g. independent random walkers, we have to replace $\gamma \rho^2$ by $\rho$ in the numerator and replace $\gamma$ by zero in the denominator. Then we exactly recover the variance of the non-stationary density fluctuation field of a system of  independent walkers starting from $\nu$.
So we see that the stickyness introduces a different time dependence of the variance visible in the extra $\lambda$-dependent terms in the denominators of 
the r.h.s. of \eqref{vardenslim}. In particular, \col{in the first term on the r.h.s. of \eqref{vardenslim} we recognize the Laplace transform of the expected local time of sticky Brownian motion}.

\subsection{Discussion}

In this section, we discuss relations to the literature, possible extensions and open problems.

\paragraph{Other applications of \col{Theorem \ref{Teo:Green}}.}
In a forthcoming work  \cite{jara} by the first and third author of this paper, in collaboration with M. Jara, 
the formula for the Laplace-Fourier transform of the transition probability of distance and sum coordinates 
is applied to obtain the second order Boltzmann Gibbs principle, which is a crucial ingredient in the proof of Kardar-Parisi-Zhang
behavior for the weakly asymmetric inclusion process.

\col{\paragraph{Dependence structure and type of convergence.} The difference and sum processes, with generators \eqref{gen-w} and \eqref{ucond}, have a dependence structure similar to their scaling limits in Theorem \ref{scalingthm}.
Namely, one process is autonomous, the other is depending on the first via a local time.
In the scaling limit one further introduces an additional time-change that however does not change
such dependence structure.
The scaling limit result in Theorem \ref{scalingthm} is proved in the sense of finite-dimensional distributions.
One could get a stronger type of convergence by directly studying the scaling limits of the generators
of the distance and sum process. This however poses additional difficulties and is not pursued here.
}

\paragraph{Scaling limits to sticky Brownian motions.} 
We observed in Remark \ref{exclusion-scaling} that exclusion particles always scale to reflected Brownian motions. In \cite{RS} R\'acz and Shkolnikov obtain multidimensional sticky Brownian motions as limits of 
exclusion processes. However this result is proved for a {\em modified} exclusion process,
in which particles slow down their velocities whenever two or more particles occupy adjacent sites.
Under diffusive scaling of space and time this slowing down results into a stickiness and the
process converge to sticky Brownian motion in the wedge  \cite{RS}.

\paragraph{Dualities.}
In this paper we have focused on self-duality of particle systems. However, the same strategy would apply to interacting diffusions that are {\em dual} to particle systems. For instance, there are processes such as  the Brownian Momentum Process \cite{gkr}, the Brownian Energy Process \cite{gkrv}  and the Asymmetric Brownian Energy Process \cite{cgrs}, which are dual to the symmetric inclusion process $\SIP(k)$. As a consequence all the results derived in this paper for the symmetric inclusion process can also be directly translated into results for these processes.

\col{\paragraph{Fluctuation field in the condensation regime.}
As far as we know, our result is the first computation dealing with the fluctuation field of 
the symmetric inclusion process in the condensation regime. We conjecture 
the expression we have found for the variance of the fluctuation field in Theorem \ref{variancethm} to have
some degree of universality within the realm of system exhibiting condensation effects.
Namely,  we believe that the scaling behaviour of the density field in the condensation regime, 
and in particular the appearance of sticky Brownian motion, is generic for systems with condensation
and goes beyond systems with self-duality, e.g. including zero range processes with condensation.
}

\col{\paragraph{Asymmetric processes.}
The main formula in Theorem \ref{Teo:Green} include interacting particle systems with asymmetries 
of certain type, for instance naive asymmetry (see Condition \ref{assumi}). One could then repeat the analysis
of the scaling limit of the two-particle process. In the weak-asymmetry limit one then expects
sticky Brownian motions with drift as limiting processes. The possibility to apply the exact formula 
to asymmetric systems with duality is instead unclear, since in the presence of naive asymmetry self-duality is lost.
One may hope to derive a more general formula for the two particle dynamics that
would apply to systems with asymmetry and self-duality such as ASEP(q,j) \cite{cgrs0}, 
ASIP(q,k) \cite{cgrs}, ABEP(k) \cite{cgrs} processes.
}

\subsection{Organization of the paper.}
The rest of this paper is organized as follows.
Section 3 contains the proof of Theorem \ref{Teo:Green} on the Laplace-Fourier transform of the 
transition probability of the distance and sum coordinates.
In Section 4 we prove Theorem \ref{scalingthm} on the scaling limits of the two particle process. 
In Section 5 we prove applications for particle systems with self-duality. 
We first prove the scaling behavior of the variance and covariances of the particle occupation 
number for the inclusion process in the condensation regime (Theorem  \ref{P:variance}). 
Then we prove the scaling
behavior for the variance of the density field in the same regime (Theorem \ref{variancethm}).

\section{Two-particle dynamics: proof of Theorem \ref{Teo:Green}}

\subsection{Outline of the proof}
The strategy to solve the two particle dynamics has two steps: first we analyze
the autonomous distance process, for which the main challenge is to treat the spatial
inhomogeneity caused by the defect in $1$; second we study the sum process by 
conditioning to the distance.

\smallskip

 Since $u(t)$ and $w(t)$ jump at the same times, we can define
the process $N(t)$  that gives the number of jumps of $(u(t),w(t))$ up to time $t\ge 0$. Notice that for any $t\ge 0$, $u(t)+w(t)\in  w+ 2 \mathbb Z$.
Given the trajectory $\{w(t), \, t\ge 0\}$, $N(t)$ is a Poisson process with time-dependent intensity which  has  the following (time-dependent) generator:
\begin{eqnarray}\label{gen}
Lf(n)=2 \(1+\frac \theta 2\, \mathbf 1_{w_t=1}\)\[f(n+1)-f(n)\].
\end{eqnarray}
\vskip.3cm
\noindent
{In the following proposition we obtain a formula for $G(w,w'\kappa,\lambda)$ in terms of the Poisson clock $N(t)$ exploiting the fact that, conditioned to the path $\{w(t),t\ge 0\}$ the process $u(t)$ performs a standard discrete-time asymmetric random walk, for which we know the characteristic function at any time.}
\bp
We have
\begin{equation}
G(w,w',\kappa,\lambda)= \int_0^\infty g_\kappa(w,w',t) \, e^{-\lambda t}\, dt,
\end{equation}
with
\begin{eqnarray}\label{g}
&& g_\kappa(w,w',t):=\mathbb E_w\[\mathbf 1_{w(t)=w'}\; \nu_\kappa  ^{N(t)}\], \qquad   \nu_\kappa  :=\cos(\kappa)-i\,(p-q) \, \sin(\kappa).
\end{eqnarray}
\ep
\bpr
We have that
\begin{eqnarray}
G(w,w',\kappa,\lambda)
&=&\sum_{u' \in \mathbb Z}e^{-i\kappa u'}\, \mathcal G\((0,w);(u',w');\lambda\)\nn\\
&=& \int_0^\infty \, e^{-\lambda t}\,\left(\sum_{u' \in \mathbb Z}P_t\((0,w);(u',w')\) \, e^{-i\kappa u'}\right) dt. \nn
\end{eqnarray}
Then we need to prove that
\begin{equation}\label{quo}
\sum_{u' \in \mathbb Z}P_t\((0,w);(u',w')\) \, e^{-i\kappa u'}=g_\kappa(w,w',t),
\end{equation}
with $g_\kappa(w,w',t)$ as in \eqref{g}.
For $\kappa\in \mathbb R$ we have
\begin{eqnarray}\label{sopra}
&&\sum_{u' \in \mathbb Z}P_t\((0,w);(u',w')\) \, e^{-i\kappa u'}=\nn\\
&& =\sum_{u' \in \mathbb Z}\mathbb E_w \Big[ P_t\Big((0,w);(u',w')\;\Big| \,  \, \{w(s), 0 \le s \le t \}\Big)\Big] \, e^{-i\kappa u'}\nn\\
&& =\mathbb E_w \Big[\mathbf 1_{w(t)=w'}\;\cdot \sum_{u' \in \mathbb Z} \mathbb{P}\Big(u(t) = u' \Big| u(0) = 0, \{w(s), 0\le s \le t\}\Big) \, \cdot e^{-i\kappa u'}\Big].
\end{eqnarray}
Let us denote by  $p^{(n)}(u,u')$  the $n$-steps transition probability function of the asymmetric discrete-time random walk 
\col{that jumps to the right with probability $p$ and to the left with probability $q=1-p$.}
Then we have
\begin{eqnarray*}
\sum_{u\in \mathbb Z} p^{(1)}\(0,u\)\, e^{-i\kappa u}={q e^{i\kappa }+p e^{-i\kappa }}=\cos(\kappa)-i (p-q) \, \sin(\kappa)=\nu_\kappa,
\end{eqnarray*}
and
\begin{eqnarray}\label{qui}
\sum_{u\in \mathbb Z} p^{(n)}\(0,u\)\, e^{-i\kappa u}=\nu_\kappa^n.
\end{eqnarray}
According to \eqref{ucond}, the conditioned  process $\{u(t) \, |\, \{w(t), t\ge 0\}\}$ is equivalent to $\{u(t)\, |\, N(t)\}$ that reduces to a discrete time random walk on $\mathbb Z$ 
\col{that jumps to the right with probability $p$ and to the left with probability $q=1-p$}.
Thus
\begin{eqnarray*}
&&\sum_{u' \in \mathbb Z}  \mathbb{P}\Big(u(t) = u' \Big| u(0) = 0, \{w(s), 0\le s \le t\}\Big)  \, \cdot e^{-i\kappa u'}\nn \\
&&=\sum_{u' \in \mathbb Z} e^{-i\kappa u'} p^{(N(t))}(0,u')\nn \\
&&=\nu_\kappa  ^{N(t)}.\nn
\end{eqnarray*}
Then \eqref{quo} follows  from \eqref{sopra}.
\epr

\noindent
{In the following we obtain a convolution equation for $g_\kappa(w,w',t)$ by conditioning on the first hitting time of the defective site 1. We distinguish several cases, depending on whether $w$ and $w'$ are equal to 0, 1, or larger than 1. When the process is at the right of 1 it can be treated as a standard random walk. This produces a system of linear equations for $G_\kappa(\cdot,\cdot,\lambda)$ that can easily be solved.  }

\subsection{Case $w'=0$}
\begin{enumerate}
\item {\bf Case $\mathbf {w\ge 2.}$}  Denote by $T_1$ the first hitting time of 1 and by $f_{T_1,w}$ \col{its probability density} when the walk  starts  from $w$. From \eqref{gen} it is clear that, for $w\ge 2$, $N(t)$ behaves as a Poisson process with rate $2$ up to time $T_1$, then
\begin{equation}\label{claim}
\mathbb E_w\[\nu_\kappa^{N(T_1)}\big| T_1\]=e^{2(\nu_\kappa-1)T_1}.
\end{equation}
Hence, denoting by $\mathcal F_{T_1}$ the pre-$T_1$ sigma-algebra of the process $w(t)$,
\begin{eqnarray}
 g_\kappa(w,0,t)&&=\mathbb E_w \Big[\mathbf 1_{w(t)=0}\;\cdot \nu_\kappa^{N(t)}\Big]\nn\\
 &&= \mathbb E_w \bigg[  \mathbb E_w \Big[\mathbf 1_{w(t)=0}\;\nu_\kappa^{N(t)}\Big| \mathcal F_{T_1}\Big]\bigg]\nn\\
 &&= \mathbb E_w \bigg[  \nu_\kappa^{N(T_1)}\;\mathbb E_w \Big[\mathbf 1_{w(t)=0}\;\nu_\kappa^{N(t)-N(T_1)}\Big| \mathcal F_{T_1}\Big]\bigg]\nn\\
 &&= \mathbb E_w \bigg[  \nu_\kappa^{N(T_1)}\;\mathbb E_1 \Big[\mathbf 1_{w(t-T_1)=0}\;\nu_\kappa^{N(t-T_1)}\Big]\bigg]\nn\\
 &&=\mathbb E_w\[ \mathbb E_w \Big[  \nu_\kappa^{N(T_1)}\;g_\kappa(1,0,t-T_1)\Big|T_1\Big]\]\nn\\
 &&=\mathbb E_w\[g_\kappa(1,0,t-T_1)\, \mathbb E_w \big[  \nu_\kappa^{N(T_1)}\big|T_1\big]\]\nn\\
 && =  \int_0^t  g_\kappa(1,0,t-s)\, f_{T_1,w}(s)\,   \mathbb E_w \big[  \nu_\kappa^{N(s)}\big|T_1=s\big] \,ds.\nn
\end{eqnarray}
As a consequence
\begin{equation}\label{conv}
g_\kappa(w,0,t)=\[(h_0 \cdot f_{T_1,w})*g_\kappa(1,0,\cdot)\](t), \quad\quad h_0(t)=  \mathbb E_w \big[  \nu_\kappa^{N(t)}\big|T_1=t\big].
\end{equation}
From the convolution equation \eqref{conv} it follows that
\begin{equation}\label{G1}
G_\kappa(w,0,\lambda)=\Psi_{w}(\lambda) \cdot G_\kappa(1,0,\lambda), \qquad \text{for any }\: w\ge 2,
\end{equation}
where
\begin{eqnarray}\label{psi0}
\Psi_{w}(\lambda)&:= &\int_0^\infty   \mathbb E_w \big[  \nu_\kappa^{N(t)}\big|T_1=t\big]\, f_{T_1,w}(t)\, e^{-\lambda t}\, dt\nn\\
&=& \mathbf E^{\text{IRW}(2)}_w\[e^{-\lambda T_1}\, \nu_\kappa^{N(T_1)}\]
\end{eqnarray}
where $ \mathbf E_w^{\text{IRW}(2)}$ is the expectation w.r. to the probability law of a symmetric  random walk in $\mathbb Z$ with hopping rate $2$, starting at time 0 from $w\ge 2$.
\item  {\bf Case $\mathbf {w=1.}$}
Let $T^{\text{ex}}_i$ be the first exit time from $i$. Then $T^{\text{ex}}_1 \sim \text{Exp}(\theta+2)$, hence
\begin{eqnarray*}
&&g_\kappa(1,0,t)= \mathbb E_1  \bigg[  \mathbb E_1 \Big[\mathbf 1_{w(t)=0}\;\cdot \nu_\kappa^{N(t)}\Big| \,\mathcal F_{T^{\text{ex}}_1}\Big]\bigg]\\
&&= \nu_\kappa\;\mathbb E_1  \bigg[  \mathbb E_{w(T^{\text{ex}}_1)} \Big[\mathbf 1_{w(t-T^{\text{ex}}_1)=0}\;\cdot \nu_\kappa^{N(t-T^{\text{ex}}_1)}\Big]\bigg]\\
&&= \nu_\kappa\;\mathbb E_1  \bigg[ g_\kappa(w(T^{\text{ex}}_1),0,t-T^{\text{ex}}_1)\bigg]
\\
&&=\nu_\kappa\;\left\{ \frac{\theta+1}{\theta+2}\, \mathbb E_1  \bigg[ g_\kappa(0,0,t-T^{\text{ex}}_1)\bigg] + \frac{1}{\theta+2}\,\mathbb E_1  \bigg[ g_\kappa(2,0,t-T^{\text{ex}}_1)\bigg] \right\}\\
&&= \nu_\kappa\;\int_0^t \left\{ \frac{\theta+1}{\theta+2}\, g_\kappa(0,0,t-s) + \frac{1}{\theta+2}\, g_\kappa(2,0,t-s) \right\} \, (\theta+2) \, e^{-(1+\gamma)s}\, ds.
\end{eqnarray*}
Thus
\begin{equation}
g_\kappa(1,0,t)=(\theta+1) \nu_\kappa [h_1*g_\kappa(0,0,\cdot)](t)+ \nu_\kappa  [h_1*g_\kappa(2,0,\cdot)](t), \qquad \text{with} \quad h_1(t):= e^{-(\theta+2)t}.\nn
\end{equation}
Then we find
\begin{equation}\label{G2}
G_\kappa(1,0,\lambda)=\frac{\nu_\kappa}{\theta+2 +\lambda}\, \[(\theta+1) G_\kappa(0,0,\lambda)+ G_\kappa(2,0,\lambda)\].
\end{equation}
\item  {\bf Case $\mathbf {w=0.}$} Now we have $T^{\text{ex}}_0 \sim \text{Exp}(2)$ then
\begin{eqnarray}
&&g_\kappa(0,0,t)=\mathbb E_0  \bigg[  \mathbb E_0 \Big[\mathbf 1_{w(t)=0}\;\cdot \nu_\kappa^{N(t)}\Big| \,\mathcal F_{T^{\text{ex}}_0}\Big]\bigg]\nn\\
&&=\mathbb E_0  \bigg[ \mathbf 1_{T^{\text{ex}}_0>t}  \,  \mathbb P_0 \Big[{w(t)=0}\Big| \,\mathcal F_{T^{\text{ex}}_0}\Big]\bigg]+\nu_\kappa\,\mathbb E_0  \bigg[ \mathbf 1_{T^{\text{ex}}_0\le t}   \, \mathbb E_1 \Big[\mathbf 1_{w(t-T^{\text{ex}}_0)=0}\;\cdot \nu_\kappa^{N(t-T^{\text{ex}}_0)}\Big]\bigg] \nn\\
&& = \mathbb P_0  \(T^{\text{ex}}_0>t\) + 2 \nu_\kappa\int_0^t  e^{-2 s}\, g_\kappa(1,0,t-s) \, ds, \nn
\end{eqnarray}
which gives
\begin{equation}
g_\kappa(0,0,t)=h_2(t)+2 \nu_\kappa[h_2*g_\kappa(1,0,\cdot)](t)\qquad \text{with} \quad h_2(t)=  e^{-2 t}.
\end{equation}
Thus
\begin{equation}\label{G3}
G_\kappa(0,0,\lambda)=\frac 1{2 +\lambda}\(1+ 2 \nu_\kappa G_\kappa(1,0,\lambda)\).
\end{equation}
\end{enumerate}
Summarizing, using \eqref{G1},  \eqref{G2} and  \eqref{G3}, we get
\begin{eqnarray}\label{1}
&& G_\kappa(0,0,\lambda)=\frac{2+\theta+\lambda-\nu_\kappa \Psi_2(\lambda)} {\nu_\kappa\mathcal Z_{\lambda,\kappa}}\,\nn \\
&& G_\kappa(1,0,\lambda)=\frac {\theta+1} {\mathcal Z_{\lambda,\kappa}}\,  \nn \\
&& G_\kappa(w,0,\lambda)=\frac{\theta+1} {\mathcal Z_{\lambda,\kappa}}\,  \Psi_{w}(\lambda) \qquad \text{for }\: w \ge 2
\end{eqnarray}
\begin{equation}
\mathcal Z_{\lambda,\kappa}=\frac1{ \nu_\kappa}\left\{(2+\lambda)(2+\theta+\lambda- \nu_\kappa \Psi_2(\lambda))-2(\theta+1) \nu_\kappa^2\right\}
\end{equation}
and $\Psi_{w}(\lambda)$ as in \eqref{psi0}.

\subsection{Case $w'=1$}
\begin{enumerate}
\item   {\bf Case $\mathbf {w\ge 2.}$} Denote by $T_1$ the first hitting time of 1 and, as before, by $f_{T_1,w}$ \col{its probability density when the walk is} starting from $w$. Then
\begin{eqnarray}
 g_\kappa(w,1,t)&&= \mathbb E_w \bigg[  \mathbb E_w \Big[\mathbf 1_{w(t)=1}\;\nu_\kappa^{N(t)}\Big| \mathcal F_{T_1}\Big]\bigg]\nn\\
 &&= \mathbb E_w \bigg[  \nu_\kappa^{N(T_1)}\;\mathbb E_1 \Big[\mathbf 1_{w(t-T_1)=1}\;\nu_\kappa^{N(t-T_1)}\Big]\bigg]\nn\\
 &&=\mathbb E_w\[g_\kappa(1,1,t-T_1)\, \mathbb E_w \big[  \nu_\kappa^{N(T_1)}\big|T_1\big]\]\nn\\
 && =  \int_0^t  g_\kappa(1,1,t-s)\, f_{T_1,w}(s)\,   \mathbb E_w \big[  \nu_\kappa^{N(s)}\big|T_1=s\big] \,ds,\nn
\end{eqnarray}
so that
\begin{equation}\label{conv2}
g_\kappa(w,1,t)=\[(h_0 \cdot f_{T_1,w})*g_\kappa(1,1,\cdot)\](t).
\end{equation}
It follows that
\begin{equation}\label{G111}
G_\kappa(w,1,\lambda)=\Psi_{w}(\lambda) \cdot G_\kappa(1,1,\lambda), \qquad \text{for any }\: w\ge 2.
\end{equation}
\item  {\bf Case $\mathbf {w=1.}$}
We have
\begin{eqnarray*}
&&g_\kappa(1,1,t)=  \mathbb E_1 \Big[\mathbf 1_{T_1^{\text{ex}}>t}\Big]+ \mathbb E_1 \Big[\mathbf 1_{T_1^{\text{ex}}\le t}\cdot\,\mathbf 1_{w(t)=1}\;\cdot \nu_\kappa^{N(t)}\Big]\nn\\
&&= \mathbb P_1 \Big({T_1^{\text{ex}}>t}\Big)+ \mathbb E_1  \bigg[  \mathbf 1_{T_1^{\text{ex}}\le t}\cdot\, \mathbb E_1 \Big[\mathbf 1_{w(t)=0}\;\cdot \nu_\kappa^{N(t)}\Big| \,\mathcal F_{T^{\text{ex}}_1}\Big]\bigg]\\
&&=e^{-(2+\theta)t}+ \nu_\kappa\;\mathbb E_1  \bigg[    \mathbf 1_{T_1^{\text{ex}}\le t}\cdot\,
 g_\kappa(w(T^{\text{ex}}_1),0,t-T^{\text{ex}}_1)\bigg]\\
&&=e^{-(2+\theta)t} +\nu_\kappa\;\left\{ \frac{(\theta+1)}{\theta+2}\, \mathbb E_1  \bigg[  \mathbf 1_{T_1^{\text{ex}}\le t}\cdot\, g_\kappa(0,0,t-T^{\text{ex}}_1)\bigg] + \frac{1}{2+\theta}\,\mathbb E_1  \bigg[ \mathbf 1_{T_1^{\text{ex}}\le t}\cdot\,  g_\kappa(2,0,t-T^{\text{ex}}_1)\bigg] \right\}\\
&&= e^{-(2+\theta)t}+\nu_\kappa\;\int_0^t \left\{{(\theta+1)}\, g_\kappa(0,0,t-s) + \, g_\kappa(2,0,t-s) \right\} \,  e^{-(2+\theta)s}\, ds.
\end{eqnarray*}
Then
\begin{equation}
g_\kappa(1,1,t)=h_1(t)+\nu_\kappa\[ h_1*\((\theta+1) g_\kappa(0,1,\cdot)+  g_\kappa(2,1,\cdot)\)\](t),
\end{equation}
hence
\begin{equation}\label{G222}
G_\kappa(1,1,\lambda)=\frac{1}{2+\theta +\lambda}\, \[1+(\theta+1) \nu_\kappa G_\kappa(0,1,\lambda)+ \nu_\kappa G_\kappa(2,1,\lambda)\].
\end{equation}

\item  {\bf Case $\mathbf {w=0.}$} Now we have $T^{\text{ex}}_0 \sim \text{Exp}(2)$. We write
\begin{eqnarray}
g_\kappa(0,1,t)&=&\mathbb E_0  \bigg[  \mathbb E_0 \Big[\mathbf 1_{w(t)=1}\;\cdot \nu_\kappa^{N(t)}\Big| \,\mathcal F_{T^{\text{ex}}_0}\Big]\bigg]\nn\\
&&=\nu_\kappa\,\mathbb E_0  \bigg[ \mathbb E_1 \Big[\mathbf 1_{w(t-T^{\text{ex}}_0)=1}\;\cdot \nu_\kappa^{N(t-T^{\text{ex}}_0)}\Big]\bigg] \nn\\
&& = 2 \nu_\kappa\int_0^t  e^{-2 s}\, g_\kappa(1,1,t-s) \, ds, \nn
\end{eqnarray}
which implies
\begin{equation}
g_\kappa(0,1,t)=2 \nu_\kappa[h_2*g_\kappa(1,1,\cdot)](t).
\end{equation}
Then
\begin{equation}\label{G333}
G_\kappa(0,1,\lambda)=\frac {2 \nu_\kappa}{2 +\lambda}\,G_\kappa(1,1,\lambda).
\end{equation}

\end{enumerate}
Thus, using \eqref{G111},  \eqref{G222} and  \eqref{G333} we get
\begin{eqnarray}\label{2}
&& G_\kappa(0,1,\lambda)=\frac{2} {\mathcal Z_{\lambda,\kappa}}\nn \\
&& G_\kappa(1,1,\lambda)=\frac {2+\lambda}  { \nu_\kappa\mathcal Z_{\lambda,\kappa}}  \nn \\
&& G_\kappa(w,1,\lambda)= \frac{2+\lambda} { \nu_\kappa\mathcal Z_{\lambda,\kappa}}\,  \Psi_w(\lambda) \qquad \text{for }\: w \ge 2.
\end{eqnarray}

\subsection{Case $w'\ge 2$}

\begin{enumerate}
\item   {\bf Case $\mathbf {w\ge 2.}$} Denoting by $T_1$ the first hitting time of 1, we have 
\begin{eqnarray}
 g_\kappa(w,w',t)&&=\mathbb E_w \Big[\mathbf 1_{T_1>t}\;\mathbf 1_{w(t)=w'}\;\cdot \nu_\kappa^{N(t)}\Big]+\mathbb E_w \Big[\mathbf 1_{T_1\le t}\;\mathbf 1_{w(t)=w'}\;\cdot \nu_\kappa^{N(t)}\Big]\nn\\
 &&=\mathbf E^{\text{IRW}(2)}_w \Big[\mathbf 1_{T_1>t}\;\mathbf 1_{w(t)=w'}\;\cdot \nu_\kappa^{N(t)}\Big]+ \mathbb E_w \bigg[  \mathbf 1_{T_1\le t}\;\mathbb E_w \Big[\mathbf 1_{w(t)=w'}\;\nu_\kappa^{N(t)}\Big| \mathcal F_{T_1}\Big]\bigg]\nn\\
 && =  \mathbf E^{\text{IRW}(2)}_w \Big[\mathbf 1_{T_1>t}\;\mathbf 1_{w(t)=w'}\;\cdot \nu_\kappa^{N(t)}\Big]+\int_0^t  g_\kappa(1,w',t-s)\, f_{T_1,w}(s)\,   \mathbb E_w \big[  \nu_\kappa^{N(s)}\big|T_1=s\big] \,ds,\nn
\end{eqnarray}
which is equivalent to
\begin{equation}\label{conv1}
g_\kappa(w,w',t)= \mathbf E^{\text{IRW}(2)}_w \Big[\mathbf 1_{T_1>t}\;\mathbf 1_{w(t)=w'}\;\cdot \nu_\kappa^{N(t)}\Big]+\[(h_0 \cdot f_{T_1,w})*g_\kappa(1,w',\cdot)\](t),
\end{equation}
where $ \mathbf E_w^{\text{IRW}(2)}$ is the expectation w.r. to the probability law of a symmetric  random walk in $\mathbb Z$ with hopping rate $2$, starting at time 0 from $w\ge 2$.
\vskip.2cm
\noindent
From the convolution equation \eqref{conv1} it follows that
\begin{equation}\label{G11}
G_\kappa(w,w',\lambda)=\Phi_{w,w'}(\lambda)+\Psi_{w}(\lambda) \cdot G_\kappa(1,w',\lambda), \qquad \text{for any }\: w\ge 2,
\end{equation}
where
\begin{equation}
\Phi_{w,w'}(\lambda):= \int_0^\infty \mathbf E^{\text{IRW}(2)}_w \Big[\mathbf 1_{T_1>t}\;\mathbf 1_{w(t)=w'}\;\cdot \nu_\kappa^{N(t)}\Big]\, e^{-\lambda t}\, dt.
\end{equation}
\item  {\bf Case $\mathbf {w=1.}$}
We have
\begin{eqnarray*}
g_\kappa(1,w',t)&&= \mathbb E_1 \Big[\mathbf 1_{w(t)=w'}\;\cdot \nu_\kappa^{N(t)}\Big]\nn\\
&&=\mathbb E_1  \bigg[\mathbb E_1 \Big[\mathbf 1_{w(t)=w'}\;\cdot \nu_\kappa^{N(t)}\Big| \,\mathcal F_{T^{\text{ex}}_1}\Big]\bigg]\\
&&=\nu_\kappa\;\mathbb E_1  \bigg[
 g_\kappa(w(T^{\text{ex}}_1),w',t-T^{\text{ex}}_1)\bigg]\\
&&=\nu_\kappa\;\left\{ \frac{\theta+1}{\theta+2}\, \mathbb E_1  \bigg[ g_\kappa(0,w',t-T^{\text{ex}}_1)\bigg] + \frac{1}{\theta+2}
\,\mathbb E_1  \bigg[  g_\kappa(2,w',t-T^{\text{ex}}_1)\bigg] \right\}\\
&&=\nu_\kappa\;\int_0^t \left\{(\theta+1)\, g_\kappa(0,w',t-s) +  g_\kappa(2,w',t-s) \right\} \,  e^{-(\theta+2)s}\, ds,
\end{eqnarray*}
i.e.
\begin{equation}
g_\kappa(1,w',t)=(\theta+1) \nu_\kappa [h_1*g_\kappa(0,w',\cdot)](t)+ \nu_\kappa [h_1*g_\kappa(2,w',\cdot)](t).
\end{equation}
Then
\begin{equation}\label{G22}
G_\kappa(1,w',\lambda)=\frac{1}{2+\theta +\lambda}\, \[(\theta+1) \nu_\kappa G_\kappa(0,w',\lambda)+  \nu_\kappa  G_\kappa(2,w',\lambda)\].
\end{equation}
\item  {\bf Case $\mathbf {w=0.}$} Now we have $T^{\text{ex}}_0 \sim \text{Exp}(2)$ then
\begin{eqnarray}
g_\kappa(0,w',t)&&=\mathbb E_0  \bigg[  \mathbb E_0 \Big[\mathbf 1_{w(t)=w'}\;\cdot \nu_\kappa^{N(t)}\Big| \,\mathcal F_{T^{\text{ex}}_0}\Big]\bigg]\nn\\
&&=\nu_\kappa\,\mathbb E_0  \bigg[ \mathbb E_1 \Big[\mathbf 1_{w(t-T^{\text{ex}}_0)=w'}\;\cdot \nu_\kappa^{N(t-T^{\text{ex}}_0)}\Big]\bigg] \nn\\
&& = 2 \nu_\kappa\int_0^t  e^{-2 s}\, g_\kappa(1,w',t-s) \, ds, \nn
\end{eqnarray}
namely
\begin{equation}
g_\kappa(0,w',t)=2 \nu_\kappa [h_2*g_\kappa(1,w',\cdot)](t).
\end{equation}
Then
\begin{equation}\label{G33}
G_\kappa(0,w',\lambda)=\frac {2 \nu_\kappa}{2 +\lambda}\,G_\kappa(1,w',\lambda).
\end{equation}

\end{enumerate}
Thus, using \eqref{G11},  \eqref{G22} and  \eqref{G33} we get
\begin{eqnarray}\label{3}
&& G_\kappa(0,w',\lambda)=\frac{2 \nu_\kappa} {\mathcal Z_{\lambda,\kappa}}\, \Phi_{2,w'}(\lambda)\nn \\
&& G_\kappa(1,w',\lambda)=\frac {2+\lambda}  {\mathcal Z_{\lambda,\kappa}}\,\Phi_{2,w'}(\lambda)  \nn \\
&& G_\kappa(w,w',\lambda)=\Phi_{w,w'}(\lambda)+ \frac {2+\lambda}  {\mathcal Z_{\lambda,\kappa}}\,\Phi_{2,w'}(\lambda)   \Psi_w(\lambda) \qquad \text{for }\: w \ge 2,
\end{eqnarray}
for $w' \ge 2$.

\subsection{Computation of $\Psi_w(\lambda)$ and $\Phi_{w,w'}(\lambda)$}
For $x\in \mathbb C$ we define
\begin{equation}\label{zeta0}
\zeta(x):=x-\sqrt{x^2-1}
\end{equation}
notice that  $\zeta(x)\in \mathbb R^+$ for any $|x|\ge 1$ and $\zeta(x)\le 1$ for any $x\in \mathbb R \cap [1,+\infty)$.
\begin{lemma}
For $w\ge 2$ we have
\begin{equation}\label{Psi}
\Psi_w(\lambda)= \zeta_{\lambda,\kappa}^{w-1}\quad \qquad \text{with}\qquad \zeta_{\lambda,\kappa}:=\zeta\(x_{\lambda,\kappa}\), \qquad x_{\lambda,\kappa}:=\frac 1 {\nu_\kappa}\(1+\frac{\lambda}{2}\)
\end{equation}
\end{lemma}
\begin{proof}
Let $\mathbf E_w^{\text{IRW}(2)}$ be the expectation with respect to a symmetric random walk on $\mathbb Z$ with hopping rate $2$, and let $S_1$ the first hitting time of 1 of the embedded discrete time random walk and denote by $\{X_i\}_{i \in \mathbb N}$ a sequence of independent  exponential random  variables of parameter $2$. Then
\begin{eqnarray}
&&\Psi_w(\lambda)=\mathbf E^{\text{IRW}(2)}_w\[e^{-\lambda T_1}\, \nu_\kappa^{N(T_1)}\] = \mathbf E_w\[ \mathbf E_w\[e^{-\lambda T_1}\, \nu_\kappa^{S_1}\Big| S_1\]\]\nn\\
&&=\sum_{n=1}^\infty \nu_\kappa^n\, \mathbf P_w\(S_1=n\)\,  \mathbf E_w\[e^{-\lambda T_1}\Big| S_1=n\]\nn\\
&&=\sum_{n=1}^\infty  \nu_\kappa^n\, \mathbf P_w\(S_1=n\)\,  \mathbf E_w\[e^{-\lambda (X_1+\ldots+X_n)}\]\nn\\
&&=\sum_{n=1}^\infty   \nu_\kappa^n\, \mathbf P_w\(S_1=n\)\,  \(\frac{2}{2+\lambda}\)^n\nn \\
&&=\mathbf E_w\[\(\frac{2 \nu_\kappa}{2+\lambda}\)^{S_1}\]=\(\zeta\(\frac{2+\lambda}{2 \nu_\kappa}\)\)^{w-1},
\end{eqnarray}
with $\zeta(x)$ as in \eqref{zeta0}.
\end{proof}

\begin{lemma}
For $w,w'\ge 2$ we have
\begin{eqnarray}\label{Phi}
\Phi_{w,w'}(\lambda)= \frac{\zeta_{\lambda,\kappa}^{|w'-w|}- \zeta_{\lambda,\kappa}^{w'+w-2}}{\nu_\kappa\(\zeta_{\lambda,\kappa}^{-1}-\zeta_{\lambda,\kappa}\)}
\end{eqnarray}
with $\zeta_{\lambda,\kappa}$ as in \eqref{Psi}.
\end{lemma}
\begin{proof}
By definition
\begin{equation}
\Phi_{w,w'}(\lambda):= \int_0^\infty \mathbf E^{\text{IRW}(2)}_w \Big[\mathbf 1_{T_1>t}\;\mathbf 1_{w(t)=w'}\;\cdot \nu_\kappa^{N(t)}\Big]\, e^{-\lambda t}\, dt
\end{equation}
\begin{eqnarray}
 \mathbf E^{\text{IRW}(2)}_w \Big[\mathbf 1_{T_1>t}\;\mathbf 1_{w(t)=w'}\;\cdot \nu_\kappa^{N(t)}\Big]=\sum_{n=0}^\infty \nu_\kappa^n\; \mathbf p_w\(S_1>n, \, w_n=w'\) \; P(N(t)=n)
\end{eqnarray}
where now $\mathbf p_w$ is the probability law of the embedded symmetric random walk on $\mathbb Z$ starting from $w\ge 2$, and $S_1$ is the related first hitting time of 1. Moreover $N(t)$ is the Poisson process of parameter $2$. From the reflection principle for the symmetric random walk we have that
\begin{eqnarray}
&&\mathbf p_w\(S_1\le n, \, w_n=w'\) =\mathbf p_0\(S_{1-w}\le n, \, w_n=w'-w\) \nn\\
&& =\mathbf p_0\(w_n=2-(w+w')\) =\mathbf p_0\(w_n=(w+w')-2\)
\end{eqnarray}
where, for $b\ge 0$,
\begin{equation}\label{Bern}
\mathbf p_0(w_n=b)=
\frac 1 {2^n} \, \binom{n}{(n+b)/2} \quad  \text{if} \:\; n\ge b\quad \text{and} \quad n+b \: \:\text{is even}\\
\end{equation}
and it is 0 otherwise. Then
\begin{eqnarray}
&& \mathbf E^{\text{IRW}(2)}_w \Big[\mathbf 1_{T_1>t}\;\mathbf 1_{w(t)=w'}\;\cdot \nu_\kappa^{N(t)}\Big]=e^{-2 t}\sum_{n=0}^\infty \;\frac{(2 \nu_\kappa t)^n}{n!}\;\left\{\mathbf p_0\(w_n=w'-w\)- \mathbf p_0\(w_n=w'+w-2\) \right\} \nn
\end{eqnarray}
hence, from \eqref{Bern} it follows
\begin{eqnarray}
&&\Phi_{w,w'}(\lambda)= \sum_{n=0}^\infty \frac{(2 \nu_\kappa)^n}{n!}\;\left\{\mathbf p_0\(w_n=w'-w\)-\mathbf p_0\(w_n=w'+w-2\) \right\}  \int_0^\infty t^n\,  e^{-(2 +\lambda)t}  \, dt \nn \\
&&=  \sum_{n=0}^\infty \frac{(2 \nu_\kappa)^n}{(2 +\lambda)^{n+1}}\;\left\{\mathbf p_0\(w_n=w'-w\)-\mathbf p_0\(w_n=w'+w-2\) \right\} \nn \\
&&= \frac 1 {2 +\lambda} \;\left\{f(w'-w)-f(w'+w-2)\right\}
\end{eqnarray}
with, for  $b\in \mathbb Z$,
\begin{eqnarray}
f(b)&:=& \sum_{n=0}^\infty \(\frac{2 \nu_\kappa}{2+\lambda}\)^n\;\mathbf p_0\(w_n=b\)= \frac{2+\lambda}{\sqrt{(2+\lambda)^2-4\nu_\kappa^2}}\, \(\zeta\(\frac{2 +\lambda}{2 \nu_\kappa}\)\)^{|b|}\nn
\end{eqnarray}
and $\zeta(x)$ as in \eqref{zeta0}.
\end{proof}

\subsection{Conclusion of the proof of Theorem \ref{Teo:Green}}
From \eqref{1},  \eqref{2} and  \eqref{3} it follows that
\begin{eqnarray}
&& G_\kappa(0,0,\lambda)=\frac{\theta \nu_\kappa^{-1}+ \zeta_{\lambda,\kappa}^{-1}} {\mathcal Z_{\lambda,\kappa}}\,\nn \\
&& G_\kappa(w,0,\lambda)=\frac{\theta+1} {\mathcal Z_{\lambda,\kappa}}\,  \zeta_{\lambda,\kappa}^{w-1}\qquad \text{for }\: w \ge 1\nn
\\
&& G_\kappa(0,w',\lambda)=\frac{2} {\mathcal Z_{\lambda,\kappa}}\, \zeta_{\lambda,\kappa}^{w'-1} \qquad \text{for }\: w' \ge 1 \nn \\
&& G_\kappa(w,w',\lambda)= \frac{\zeta_{\lambda,\kappa}^{|w'-w|}- \zeta_{\lambda,\kappa}^{w'+w-2}}{\nu_\kappa\(\zeta_{\lambda,\kappa}^{-1}-\zeta_{\lambda,\kappa}\)}+ \frac {2 x_{\lambda,\kappa}}  {\mathcal Z_{\lambda,\kappa}}\, \zeta_{\lambda,\kappa}^{w'+w-2} \qquad \text{for }\: w,w' \ge 1
\end{eqnarray}
%
%
with
\begin{equation}\label{zeta2}
 \zeta_{\lambda,\kappa}:=\zeta\(x_{\lambda,\kappa}\)=x_{\lambda,\kappa}-\sqrt{x_{\lambda,\kappa}^2-1}, \qquad x_{\lambda,\kappa}:=\frac 1 {\nu_\kappa}\(1+\frac \lambda {2}\)
\end{equation}
and
\begin{eqnarray}
\mathcal Z_{\lambda,\kappa}&=&\frac 1{\nu_\kappa}\left\{(2+\lambda)(2+\theta+\lambda- \nu_\kappa \Psi_2(\lambda))-2(\theta+1) \nu_\kappa^2\right\}\nn \\
&=& \(\nu_\kappa(\zeta_{\lambda,\kappa}^{-2}-1)+2\theta \(x_{\lambda,\kappa}-\nu_\kappa\)\)\nn \\
&=& \mathcal Z_{\lambda,\kappa}^{(0)}+2\theta  (x_{\lambda,\kappa}-\nu_\kappa), \qquad  \mathcal Z_{\lambda,\kappa}^{(0)}:= \nu_\kappa (\zeta_{\lambda,\kappa}^{-2}-1)
\end{eqnarray}
Notice that, for $w,w'\ge 1$
\begin{eqnarray}
G_\kappa(w,w',\lambda)&=
& \frac {1} {\mathcal Z^{(0)}_{\lambda,\kappa}}
\left\{ \zeta_{\lambda,\kappa}^{|w'-w|-1}+\zeta_{\lambda,\kappa}^{w'+w-1} \cdot
\left( 2\,  \frac{\mathcal Z_{\lambda,\kappa}^{(0)} }{\mathcal Z_{\lambda,\kappa}}-1 \right)\right\} \nn
\end{eqnarray}
thus
 \begin{eqnarray}
&& G_\kappa(0,0,\lambda)=\frac{\theta \nu_\kappa^{-1}+  \zeta_{\lambda,\kappa}^{-1}} {\mathcal Z_{\lambda,\kappa}}\,\nn \\
&& G_\kappa(w,0,\lambda)=\frac{ \theta +1} {\mathcal Z_{\lambda,\kappa}}\,  \zeta_{\lambda,\kappa}^{w-1}\qquad \text{for }\: w \ge 1\nn
\\
&& G_\kappa(w,w',\lambda)= \frac {1} {\mathcal Z^{(0)}_{\lambda,\kappa}}
\left\{ \zeta_{\lambda,\kappa}^{|w'-w|-1}+\zeta_{\lambda,\kappa}^{w'+w-1} \cdot
\left( 2\,  \frac{\mathcal Z_{\lambda,\kappa}^{(0)} }{\mathcal Z_{\lambda,\kappa}}-1 \right)\right\}  \qquad \text{for }\: w\ge 0 ,w' \ge 1.\nn
\end{eqnarray}
This finishes the proof of the Theorem \ref{Teo:Green}.
\qed

\subsection{Proof of Corollary \ref{for}}

Notice that by the definition of the coordinates of the leftmost and rightmost particle one has
$w=y-x$ and  $u=x+y$ so that, as a consequence of \eqref{GL} one obtains
\begin{equation}
\label{prima}
\Pi^{(\theta)}((x,y),(x',y');\lambda)=\mathcal{G}^{(\theta)}((x+y,y-x),(x'+y',y'-x');\lambda)\;.
\end{equation}
To obtain an explicit expression we use  translation invariance 
\begin{equation}
\mathcal G((u,w);(u',w');\lambda)=\mathcal G((0,w);(u'-u,w');\lambda)
\end{equation}
and we rewrite $\mathcal G^{(\theta)}$ as follows
\be
\label{seconda}
\mathcal G^{(\theta)}((u,w);(u',w');\lambda)=\frac 1 {2\pi}\int_{-\pi}^{\pi} G^{(\theta)}(w,w',\kappa,\lambda) e^{i\kappa (u'-u)}\, d\kappa,
\ee
where on the right hand side we can insert the expression for $G^{(\theta)}$ that appears in Theorem \ref{Teo:Green}.
In doing so it is also useful to write an integral representation for the terms $\zeta_{\lambda,\kappa}^{|w'-w|}$ and $\zeta_{\lambda,\kappa}^{w'+w}$
in Theorem \ref{Teo:Green}. This can be obtained by noticing that for any $\zeta$ with $|\zeta|<1$ one has
\begin{eqnarray}
\sum_{x=0}^\infty \zeta^x \cos{(mx)}=\frac {1-\zeta\cos(m)}{1+\zeta^2 -2\zeta\cos(m) }
\end{eqnarray}
so that
\begin{eqnarray}
\sum_{x\in \mathbb Z} \zeta^{|x|}e^{-imx}= \frac{1-\zeta^2}{1+\zeta^2-2\zeta\cos(m)}\;.
\end{eqnarray}
Thus we have
\begin{eqnarray}\label{terza}
\zeta^{|x|}
&=& \frac 1{2\pi}\int_{-\pi}^\pi e^{ imx}\, \frac{\zeta^{-1}-\zeta}{\zeta^{-1}+\zeta-2 \cos(m)}\, dm\;.
\end{eqnarray}
We are indeed allowed to use this expression since $|\zeta_{\lambda,\kappa}|\le 1$, as it can
immediately be seen from \eqref{zeta} observing that  $x_{\lambda,\kappa}\ge 1$ for all $\alpha, \lambda\ge 0$.
All in all, combining together Eq. \eqref{seconda}, Theorem \ref{Teo:Green} and Eq. \eqref{terza}, one arrives to
 \begin{eqnarray}
 \label{platone}
 &&\mathcal G^{(\theta)}((u,w);(u',w');\lambda)=
 \\\nn
\\
&&=
\frac {1}{8\pi^2}\int_{-\pi}^\pi \, \int_{-\pi}^\pi \, \frac{f_{\lambda,\kappa}^{(\theta)}(w,w')\, e^{i\kappa (u'-u)}}{\(1+\frac \lambda{2}\)- \nu_\kappa\cos(m)} \(e^{ im(w'-w)} + e^{ im(w'+w)} \cdot
\left( 2\,  \frac{\mathcal Z_{\lambda,\kappa}^{(0)} }{\mathcal Z^{(\theta)}_{\lambda,\kappa}}-1 \right)\) \, dm\, d\kappa \;.\nn
\end{eqnarray}
This expression can now be used  in Eq. \eqref{prima}. Defining \begin{eqnarray}
&&\kappa_1=\kappa-m, \quad \kappa_2= \kappa+m,
\end{eqnarray}
and considering the case  $x\neq y$ one obtains \eqref{left-right-formula}.
\qed

\section{Scaling limits: proof of Theorem \ref{scalingthm}}


\col{This section is organised as follows: we first prove several results on the limiting process for the properly rescaled joint process of distance and sum coordinates,
and then we use these computations to show that the candidate scaling limit is the ``correct scaling limit''.}

\col{
Specifically, in the preliminary section \ref{preliminary-sect}  we compute the Fourier-Laplace transform of the probability density of the standard sticky Brownian motion.
In section \ref{joint-section} we obtain the Fourier-Laplace transform of the transition density of the candidate scaling limit, i.e. a system of two Brownian motions with a 
sticky interaction defined via their local intersection time. Further, in section \ref{limiting} we show that the Fourier-Laplace transform for a generic
value of the stickiness parameter can be written as a convex combination of the  Fourier-Laplace transforms of two
limiting cases, i.e. absorbing and reflecting. 
With these results in our hand, we then continue by proving in section \ref{scaling-section} convergence of the appropriately rescaled discrete two-particle process to 
the system of sticky Brownian motions (Theorem \ref{scalingthm}). The convergence in distribution is inferred from the convergence of the Fourier-Laplace transforms.  
Finally section \ref{local-sect} contains the proof of Corollary \ref{local0}, that deals with the convergence of the expected local time.}

\subsection{\col{Standard} Brownian motion sticky at the origin}
\label{preliminary-sect}
\col{We start with a preliminary computation that involves just a single sticky Brownian motion
(which is indeed the scaling limit of the distance process).
We recall the definition of the sticky Brownian motion, 
see \cite{ks, pot} for more background on such process.
For all $t\ge 0$,} let $L(t)$ be the local time at the origin of a standard Brownian motion $B(t)$  and let $\gamma>0$.
Set
\be
\label{s-transform}
s^{-1}(t) = t + \gamma L(t).
\ee
The (one-sided) \col{standard} sticky Brownian motion $B^S(t)$ on $\mathbb{R}_+$ with sticky boundary at the origin
and stickiness parameter $\gamma>0$ is
defined as the time changed \col{standard} reflected Brownian motion, i.e.
\be
\label{bbb}
B^S(t) = |B(s(t))|.
\ee
Using  the expression for the joint density of $(|B(t)|, L(t))$ (formula (3.14) in \cite{pot})
\be\label{dens}
\mathbb{P}_0(|B(r)| \in dx, L(r)\in dy) = 2 \frac{x+y}{\sqrt{2\pi r^3}} e^{-\frac{(x+y)^2}{2r}} dx dy \qquad x,y \ge 0,
\ee
\col{one can}  compute the Fourier-Laplace transform of $B^S(t)$ \col{that is defined as}
\be
\label{psi-single}
\psi_{\col{0}}^{S}(m,\lambda,\gamma):=\int_0^\infty\E_{0} \left[e^{-im B^S(t)}\right] e^{-\lambda t} dt,
\ee
\col{where the subscript $0$ denotes the initial position of $B^S(t)$.}

\bl[\col{Fourier-Laplace transform of standard sticky Brownian motion}]
\label{lemma-preliminary}
We have
\be\label{sticky}
\psi_{\col{0}}^{S}(m,\lambda,\gamma) = \frac{\col{\sqrt{2}}}{\sqrt{\lambda} + \frac{\gamma}{\sqrt{2}}\lambda}
\int_{0}^{\infty} e^{-\sqrt{2\lambda}x - imx}dx
+ \frac{\frac{\gamma}{\sqrt{2}}}{\sqrt{\lambda} + \frac{\gamma}{\sqrt{2}}\lambda}.
\ee
\el
\bpr
\col{We  rewrite \eqref{psi-single} using \eqref{bbb}  and then} apply the change of variable $s(t) = r$ to obtain
\beq
\psi^{S}_0(m,\lambda,\gamma) - \frac{1}{\lambda} &=& \E_{0} \left[ \int_0^\infty\left(e^{-im B^S(t)} -1 \right)e^{-\lambda t} dt \right]\nn\\
&=&\E_{0} \left[ \int_0^\infty\left(e^{-im |B(r)|} -1 \right)e^{-\lambda (r+\gamma L(r))} (dr+\gamma dL(r)) \right].
\eeq
The local time $L(t)$ only grows when $B(t)$ is at the origin implying that the term into the round bracket is zero in the integral with respect to $dL(r)$. As a consequence we have
\be
\psi^{S}_0(m,\lambda,\gamma) - \frac{1}{\lambda} = \E_{0} \left[ \int_0^\infty\left(e^{-im |B(r)|} -1 \right)e^{-\lambda (r+\gamma L(r))} dr \right].
\ee
Then  \eqref{sticky} follows by using  the expression \eqref{dens} and the formula for the Laplace transform
\be\label{Lapl}
\int_{0}^{\infty} e^{-\lambda r} \frac{a}{\sqrt{2\pi r^3}} e^{-a^2/2r} dr = e^{-\sqrt{2\lambda} a} \quad \qquad a>0.
\ee
\epr

\subsection{The joint sticky process}
\label{joint-section}
Let $\tilde B(t)$ and $B(t)$ be two independent Brownian motions starting, respectively,  from $0$ and from $z\ge 0$. Let $s(t)$ \col{be defined via \eqref{s-transform}}, with $L(t)$ \col{being now the} local time of $B(t)$.
\col{We compute in this section \col{the} Fourier-Laplace transform of the candidate scaling limit, i.e.
the joint  process $(\tilde B(2t-s(t)),|B(s(t))|)$, that is defined as}
\be
\label{psi-sticky}
\Psi_z^{S}(\kappa,m,\lambda,\gamma):=\int_0^\infty\E_{0,z}\left[e^{-i\kappa \tilde B(2t-s(t))-im |B(s(t))|}\right] e^{-\lambda t} dt,
\ee
\col{where the expectation $\mathbb{E}_{0,z}$ denotes expectation w.r.t. both the  $\tilde B(t)$ Brownian motion that starts from $0$ and the $B(t)$ process that starts from $z\ge0$.}

\col{We start with the following Lemma that extends \eqref{dens} to a positive initial condition.}
\bl[\col{Joint density of reflected Brownian motion and local time}] 
\col{For all $z > 0$} we have
\begin{eqnarray}
\label{dens2}
\mathbb{P}_z(|B(t)| \in dx, L(t)\in dy) & = &  
\frac 1 {\sqrt{2\pi t}} \cdot \(e^{- \frac{(z-x)^2}{2t}}- e^{- \frac{(z+x))^2}{2t}} \)\, \delta_0(y)\,dx\, dy\\
&+ & \left(2\int_0^t   \frac{x+y}{\sqrt{2\pi (t-s)^3}} \;e^{-\frac{(x+y)^2}{2(t-s)}} \cdot\frac{z}{\sqrt{2\pi s^3}}\; e^{-\tfrac{z^2}{2s}} \, ds\right)  dx \, dy\nn
\end{eqnarray}
\el
\bpr
Let $\nu_z(\cdot)$ be the probability density function of $\tau_z$, the first hitting time of \col{$0$} for a Brownian motion starting from \col{$z > 0$}, i.e.
\be
\label{hit}
\col{\nu_z(s)=\frac{z}{\sqrt{2\pi s^3}}\, e^{-\tfrac{z^2}{2s}}} .
\ee
\col{By conditioning to the time $t$ being smaller or larger than $\tau_z$ we have}
\beq
\mathbb{P}_z(|B(t)| \in dx, L(t)\in dy) & = & \mathbb{P}_z(|B(t)|\in dx, \, \min_{s\le t}B(s)>0)\cdot \delta_0(y) \, dy\\
&+ &\int_0^t \mathbb{P}_0(|B(t-s)| \in dx, L(t-s)\in dy) \, \col{\nu_{z}(s) ds}. \nn
\eeq
\col{The reflection principle for Brownian motion gives}
\beq
 \mathbb{P}_z\(|B(t)|\in dx, \, \min_{s\le t}B(s)>0\)
=\frac 1 {\sqrt{2\pi t}} \cdot \(e^{- \frac{(z-x)^2}{2t}}- e^{- \frac{(z+x)^2}{2t}} \)\, dx, \nn
\eeq
\col{whereas the use of \eqref{dens} and \eqref{hit} yields}
\beq
&&\int_0^t \mathbb{P}_0(|B(t-s)| \in dx, L(t-s)\in dy) \, \col{\nu_{z}(s) ds}\nn\\
&&= \left(2\int_0^t   \frac{x+y}{\sqrt{2\pi (t-s)^3}} \;e^{-\frac{(x+y)^2}{2(t-s)}} \cdot\frac{z}{\sqrt{2\pi s^3}}\; e^{-\tfrac{z^2}{2s}} \, ds \right) dx \, dy.
\eeq
This concludes the proof.
\epr

\col{Armed with the previous Lemma we can compute the Fourier-Laplace transform defined in \eqref{psi-sticky}}.
\bl[\col{Fourier-Laplace transform of the joint sticky process}] For all $z\ge 0$ we have
\beq
&&\Psi_{\col{z}}^{S}(\kappa,m,\lambda,\gamma)=\frac {\gamma \,  e^{-\sqrt{\kappa^2+2\lambda}z}}{\gamma(\kappa^2+\lambda)+\sqrt{\kappa^2+2\lambda}} \; + \nn\\
&&\frac 1 {\sqrt{\kappa^2+2\lambda}}\,\bigg\{\frac{\sqrt{\kappa^2+2\lambda}-\gamma(\kappa^2+\lambda)}{\sqrt{\kappa^2+2\lambda}+\gamma(\kappa^2+\lambda) }\,\int_0^\infty  e^{-imx}\,e^{-\sqrt{\kappa^2+2\lambda}|z+x|} \,dx+\; \int_0^\infty\,e^{-imx} e^{-\sqrt{\kappa^2+2\lambda}|z-x|}\, dx\bigg\}.\nn
\eeq
\el

\bpr
\col{We follow a strategy similar to the one in the proof of Lemma \ref{lemma-preliminary}. It is convenient to write}
\be
\Psi_z^{S}(\kappa,m,\lambda,\gamma) - f_z(\kappa,\lambda)=  \int_0^\infty \col{\E_{0,z}} \left[ e^{-i\kappa \tilde B(2t-s(t))}\left(e^{-im B(s(t))} -1 \right)\right] e^{-\lambda t} dt,  \nn
\ee
where 
$$
f_z(\kappa,\lambda)=\int_0^\infty \col{\E_{0,z}}\left[  e^{-i\kappa \tilde B(2t-s(t))} \right] e^{-\lambda t} dt.
$$
and $\E_{0,z}$ denotes expectation with respect to the $\tilde B(t)$ process started at $0$ and the $B(t)$ process started at $z$. 
We apply the change of variable $s(t) = r$ to obtain
\beq
\Psi_z^{S}(\kappa,m,\lambda,\gamma) - f_z(\kappa,\lambda)&=& \col{\E_{0,z}} \left[ \int_0^\infty e^{-i\kappa \tilde B(r+2\gamma L(r))}\left(e^{-im |B(r)|} -1 \right)e^{-\lambda (r+\gamma L(r))} (dr+\gamma dL(r)) \right]\nn\\
&=& \col{\E_{0,z}} \left[ \int_0^\infty e^{-i\kappa \tilde B(r+2\gamma L(r))}\left(e^{-im |B(r)|} -1 \right)e^{-\lambda (r+\gamma L(r))} dr \right],
\eeq
\col{where the last equality uses again that the local time $L(t)$ only grows when $B(t)$ is at the origin.
Using then the independence of $\tilde B(t)$ and $B(t)$, and the expression for the characteristic function of the standard Brownian motion, we arrive to}
\beq
\Psi_z^{S}(\kappa,m,\lambda,\gamma) &=& \E_z \left[ \int_0^\infty e^{-\frac{\kappa^2}2 (r+2\gamma L(r))}\,e^{-im |B(r)|} \,e^{-\lambda (r+\gamma L(r))} dr \right]
 \nn\\
&+&f_z(\kappa,\lambda)-\E_z \left[ \int_0^\infty e^{-\frac{\kappa^2}2 (r+2\gamma L(r))}e^{-\lambda (r+\gamma L(r))} dr \right] \label{equa}
\eeq
\col{We now evaluate separately the two terms on the r.h.s.. For the first term, thanks to the formula \eqref{dens2}, we may write}

\beq
&&\E_z \left[ \int_0^\infty e^{-\frac{\kappa^2}2 (r+2\gamma L(r))}\,e^{-im |B(r)|} \,e^{-\lambda (r+\gamma L(r))} dr \right]=\nn\\
&& \int_0^\infty e^{-imx} \left[ \int_0^\infty   e^{-(\frac{\kappa^2}2+\lambda)r} \frac 1 {\sqrt{2\pi r}}\, \(e^{- \frac{(z-x)^2}{2r}}- e^{- \frac{(z+x)^2}{2r}} \)\, \,dr \right]dx+ \label{undo}\\
&&\int_0^\infty\,e^{-imx}\int_0^\infty e^{-(\kappa^2+\lambda)\gamma y} \left[ \int_0^\infty e^{-(\frac{\kappa^2}2+\lambda)r} \left (2 \int_0^r  \frac{x+y}{\sqrt{2\pi (r-s)^3}} \;e^{-\frac{(x+y)^2}{2(r-s)}} \cdot\frac{z}{\sqrt{2\pi s^3}}\; e^{-\tfrac{z^2}{2s}} \, ds\right)  dr \right]\,  dy \, dx . \nn \\ \label{dude}
\eeq
\col{In \eqref{undo} we may use the formula for the Laplace transform
\be
\label{lapl-gauss}
\int_{0}^{\infty} e^{-\lambda r} \frac{1}{\sqrt{2\pi r}} e^{-a^2/2r} = \frac{e^{-a \sqrt{2\lambda}}}{\sqrt{2 \lambda}} ,
\ee
and in \eqref{dude} we may employ the Laplace transform \eqref{Lapl} and the convolution rule.
All in all, we find}
\beq
&&\E_z \left[ \int_0^\infty e^{-\frac{\kappa^2}2 (r+2\gamma L(r))}\,e^{-im |B(r)|} \,e^{-\lambda (r+\gamma L(r))} dr \right]=\nn\\
&&\frac 1 {\sqrt{\kappa^2+2\lambda}}\,\int_0^\infty  e^{-imx}\,\(e^{-|z-x|\sqrt{\kappa^2+2\lambda}}-e^{-|z+x|\sqrt{\kappa^2+2\lambda}}\) \,dx+\nn\\
&&\frac {2} {\gamma(\kappa^2+\lambda) +\sqrt{\kappa^2+2\lambda}}\; \int_0^\infty\,e^{-imx} e^{-\sqrt{\kappa^2+2\lambda}(x+z)}\, dx .\label{home1}
\eeq
\col{For the second term on the r.h.s. of \eqref{equa} we observe that}
\beq
 f_z(\kappa,\lambda)&=&\col{\E_{0,z}} \left[ \int_0^\infty e^{-i\kappa \tilde B(2t-s(t))}e^{-\lambda t} dt \right]\nn\\
&=&\E_z \left[ \int_0^\infty e^{-\frac {\kappa^2}2(r+2\gamma L(r))}e^{-\lambda (r+\gamma L(r))} d(r+\gamma L(r)) \right].\nn
\eeq
\col{As a consequence we find}
\be
f_z(\kappa,\lambda)-\E_z \left[ \int_0^\infty e^{-\frac{\kappa^2}2 (r+2\gamma L(r))}e^{-\lambda (r+\gamma L(r))} dr \right]
=\gamma\E_z \left[ \int_0^\infty e^{-\frac {\kappa^2}2(r+2\gamma L(r))}e^{-\lambda (r+\gamma L(r))} dL(r)\right] .\nn
\ee
\col{This last integral can be evaluated integrating by parts, yielding}
\beq
&&\gamma\E_z \left[ \int_0^\infty e^{-(\frac {\kappa^2}2+\lambda)r}\, e^{-\gamma (\kappa^2+\lambda) L(r)} dL(r)\right]\nn\\
&&=- \frac 1{\kappa^2+\lambda}\,\E_z \left[ \int_0^\infty e^{-(\frac {\kappa^2}2+\lambda)r}\, d \col{\left ( e^{-\gamma (\kappa^2+\lambda) L(r)} \right )} \right]\nn\\
&&= \frac 1{\kappa^2+\lambda}\;\left\{\E_z \left[ \int_0^\infty d \col{\left( e^{-(\frac {\kappa^2}2+\lambda)r}\right)} \, e^{-\gamma (\kappa^2+\lambda) L(r)} \right]+
\E_z \left[ e^{-\gamma (\kappa^2+\lambda) L(0)} \right]\right\}\nn\\
&&= \frac 1{\kappa^2+\lambda}\,\left\{-\(\frac{\kappa^2}2+\lambda\)\E_z \left[ \int_0^\infty  e^{-(\frac {\kappa^2}2+\lambda)r}\, e^{-\gamma (\kappa^2+\lambda) L(r)}\, dr\right]+1\right\}\nn\\
&&=- \frac {\tfrac{\kappa^2}2+\lambda}{\kappa^2+\lambda}\;\E_z \left[ \int_0^\infty  e^{-(\frac {\kappa^2}2+\lambda)r}\, e^{-\gamma (\kappa^2+\lambda) L(r)}\, dr\right]+\frac 1 {\kappa^2+\lambda}.\nn
\eeq
\col{Furthermore, by using formula \eqref{dens2}, we have}
\beq
&&\E_z \left[ \int_0^\infty  e^{-(\frac {\kappa^2}2+\lambda)r}\, e^{-\gamma (\kappa^2+\lambda) L(r)}\, dr\right] = \nn\\
&& \int_0^\infty\int_0^\infty \frac 1 {\sqrt{2\pi r}} \; e^{-(\frac {\kappa^2}2+\lambda)r} \cdot \(e^{- \frac{(z-x)^2}{2r}}- e^{- \frac{(z+x)^2}{2r}} \)\, dr\,dx+\nn\\
&&2 \int_0^\infty \int_0^\infty e^{-\gamma (\kappa^2+\lambda) y} \col{\left[ \int_0^\infty e^{-(\frac {\kappa^2}2+\lambda)r} \left(\int_0^r   \frac{x+y}{\sqrt{2\pi (r-s)^3}} \;e^{-\frac{(x+y)^2}{2(r-s)}} \frac{z}{\sqrt{2\pi s^3}}\; e^{-\tfrac{z^2}{2s}} \, ds \right )dr \right]} dx \, dy = \nn\\
&&\frac 1 {\sqrt{\kappa^2+2\lambda}}\, \int_0^\infty \(e^{-|z-x|\sqrt{\kappa^2+2\lambda}}-e^{-|z+x|\sqrt{\kappa^2+2\lambda}}\) \,dx+\frac {2\,  e^{-\sqrt{\kappa^2+2\lambda}z}}{\gamma(\kappa^2+\lambda)+\sqrt{\kappa^2+2\lambda}} \cdot \frac 1 {\sqrt{\kappa^2+2\lambda}}, \nn
\eeq
\col{where in the last equality we used again the Laplace transforms \eqref{Lapl} and \eqref{lapl-gauss}.
Summarizing, for the second term on the r.h.s. of \eqref{equa} we have}
\beq
&&f_z(\kappa,\lambda)-\E_z \left[ \int_0^\infty e^{-\frac{\kappa^2}2 (r+2\gamma L(r))}e^{-\lambda (r+\gamma L(r))} dr \right]\nn\\
&&=-\frac 1 2 \, \frac {\sqrt{{\kappa^2}+2\lambda}}{\kappa^2+\lambda}\int_0^\infty \(e^{-|z-x|\sqrt{\kappa^2+2\lambda}}-e^{-|z+x|\sqrt{\kappa^2+2\lambda}}\) \,dx\nn\\
&&- \frac {\sqrt{{\kappa^2}+2\lambda}}{\kappa^2+\lambda}
\frac {\,  e^{-\sqrt{\kappa^2+2\lambda}(z-1)}}{\gamma(\kappa^2+\lambda)+\sqrt{\kappa^2+2\lambda}} +\frac 1 {\kappa^2+\lambda}. \label{home2}
\eeq
\col{Hence, substituting \eqref{home1} and \eqref{home2} into \eqref{equa}, the statement of the Lemma follows after elementary simplifications.}
\epr

\subsection{Limiting cases}
\label{limiting}
\col{In this section we show that the joint sticky process interpolates between two limiting cases.}
\subsubsection{Reflection: $\gamma=0$}
\col{Let $B(t)$ be a Brownian motion starting from $0$ and $B^R(t)$ be a Brownian motion on $\mathbb R^+$ reflected at 0 and starting from $z\ge 0$. Suppose they are independent and denote by  $\Psi_z^R$ the Laplace transform of the characteristic function of the joint process, i.e. 
\be
\Psi^R_{z}(\kappa,m,\lambda):=\int_0^\infty\E_{0,z} \left[e^{-i\kappa \tilde B(t) -i m B^R(t)}\right] e^{-\lambda t} dt. 
\ee
Then we have 
\be\label{PsiR}
\Psi^R_{z}(\kappa,m,\lambda)=\frac{1}{\sqrt{2\lambda+{\kappa^2}}}\, \int_0^\infty  e^{-imx} \,
\left(e^{-\sqrt{2\lambda+\kappa^2}|x+z|}+e^{-\sqrt{2\lambda+\kappa^2}|x-z|} \right) \, dx.
\ee
Indeed, from the knowledge of the transition probability of the joint process, an immediate computation gives
\beq
&&\int_0^\infty\E_{0,z} \left[e^{-i(\kappa \tilde B(t) + m B^R(t))}\right] e^{-\lambda t} dt =
\frac{1}{2\pi}\int_0^\infty  e^{-imx} \int_{-\infty}^{\infty} \frac{e^{i\bar m(x+z)}+e^{i\bar m(x-z)}}{\lambda +\frac{\kappa^2}{2}+\frac{\bar m^2}{2} }\, d\bar m \, dx.\nn
\eeq
This yields \eqref{PsiR} using the Fourier transform 
\be
\label{Fourier-formula}
\int_{-\infty}^{\infty}\frac{e^{ika}}{k^2+b^2}\, dk=\frac \pi b \, e^{-b|a|}.
\ee.
}
\subsubsection{Absorption: $\gamma\to \infty$}
\col{Let  $\tilde B(t)$ be a standard Brownian motion and 
let  $B^A(t)$ be an independent  Brownian motion on $\mathbb R^+$ 
absorbed in 0 and starting from $z > 0$. Define $\tau_z$ as the absorption time,
i.e. $\tau_z = \inf\{t \ge 0 : B^A(t) = 0\}$}.  

\col{
The process $\tilde B(2t-t\wedge \tau_z)$ describes the evolution of the sum of two coalescing Brownian motions started at two positions such that initially the sum is zero. Indeed, given the coalescing time $\tau_z$ (i.e. the hitting time of level zero for the distance of the two coalescing Brownian motions when the distance is initially $z>0$), the center of mass of the two coalescing Brownian motion evolves as the sum of two independent standard Brownians until the coalescing time and, after coalescence, it evolves as a Brownian started at the position where coalescence occurred with double speed.
Define
\be
\Psi^A_{z}(\kappa,m,\lambda)=\int_0^\infty\E_{0,z} \left[e^{-i\kappa \tilde B(2t-t\wedge \tau_z) -im B^A(t)}\right] e^{-\lambda t} dt,
\ee
\noindent
 then we have
\be\label{PsiA}
\Psi^A_{z}(\kappa,m,\lambda)=\frac{1}{\sqrt{2\lambda+{\kappa^2}}}\, \int_0^\infty  e^{-imx'} \,
\left(e^{-\sqrt{2\lambda+\kappa^2}|x-z|}-e^{-\sqrt{2\lambda+\kappa^2}|x+z|} \right) \, dx+\frac{e^{-z \sqrt{2\lambda +\kappa^2}}}{\lambda+\kappa^2}.
\ee
To show this one starts from
\beq 
\Psi^A_{z}(\kappa,m,\lambda)
=\int_0^\infty  \E_{0,z} \left[e^{-i\kappa \tilde B(t) -im B^A(t)} \mathbf{1}_{t < \tau_z }\right] \, e^{-\lambda t} dt + \int_0^\infty\E_{0} \left[e^{-i\kappa \tilde B(2t-\tau_z)} \mathbf{1}_{t \ge \tau_z }\right] e^{-\lambda t} dt. \nn \\ \label{sopraa}
\eeq
The first term in the r.h.s. of \eqref{sopraa} is given by
\beq
&&
\frac{1}{2\pi}\int_0^\infty  e^{-imx} \int_{-\infty}^{\infty} \frac{e^{i\bar m(x-z)}-e^{i\bar m(x+z)}}{\lambda +\frac{\kappa^2}{2}+\frac{\bar m^2}{2} }\, d\bar m \, dx\nn\\
&&= \int_0^\infty  e^{-imx} \, \frac{1}{\sqrt{2\lambda+{\kappa^2}}}
\left(e^{-\sqrt{2\lambda+\kappa^2}|x-z|}-e^{-\sqrt{2\lambda+\kappa^2}|x+z|} \right) \, dx.\nn
\eeq
For the second term in the r.h.s of \eqref{sopraa}, let $\nu_z(\cdot)$
be the probability density of $\tau_z$ (cf. \eqref{hit}). Then, using the fact that
\be
\label{integra}
\int_0^\infty e^{-\lambda t}\, \nu_z(t)\, dt= e^{-z \sqrt{2\lambda}},
\ee
we obtain that  the second term in the r.h.s of \eqref{sopraa} is equal to
\beq
\int_0^\infty dt \, e^{-\lambda t}\int_0^t ds\, \nu_z(s)\int_{-\infty}^{\infty}dy'\, e^{-i\kappa y'}\int_{-\infty}^{\infty} dy \,
\frac{e^{-\frac{y^2}{2s}}}{\sqrt{2\pi s}}\, \frac{e^{\frac{-(y'-y)^2}{4(t-s)}}}{\sqrt{4\pi(t-s)}}=\frac{e^{-z \sqrt{2\lambda +\kappa^2}}}{\lambda+\kappa^2}
\eeq
This is obtained by first doing the integrals in $y$ and $y'$ as Fourier transforms of suitable Brownian kernels
and then by  applying integration by parts to the $dt$ integral followed by the use of formula \eqref{integra}.
}

\subsubsection{Summary}
\col{
Notice that we can rewrite the function $\Psi_z^{S}$ as an interpolation between $\Psi_z^{R}$ and $\Psi_z^A$ as follows
\be
\label{PsiS}
\Psi^{S}_z(\kappa,m,\lambda,\gamma) =
c^{(\gamma)}(\kappa,\lambda)\;
\; \Psi^R_{z}(\kappa, m,\lambda) +\;
\left(1-c^{(\gamma)}(\kappa,\lambda)\right)
\;\Psi^A_{z}(\kappa, m,\lambda)
\ee
with
\be\label{c}
c^{(\gamma)}(\kappa,\lambda) = \frac {\sqrt{\kappa^2+2\lambda }}{\sqrt{\kappa^2+2\lambda }+\gamma\(\kappa^2+ \lambda\)}
\ee
Notice also that
\begin{equation}
\Psi^{S}_z(\kappa, m,\lambda,0)  =
\; \Psi^R_z(\kappa, m,\lambda) \qquad \text{and} \qquad
\lim_{\gamma\to+\infty}\Psi^{S}_{z}(\kappa, m,\lambda,\gamma)   =
\; \Psi^A_z(\kappa,m,\lambda)
\end{equation}
since  $c^{(0)}(\kappa,\lambda) = 1$  and $\lim_{\gamma \to \infty}c^{(\gamma)}(\kappa,\lambda) = 0$.
}

\subsection{Scaling limit}
\label{scaling-section}

\bp[\col{Convergence of Fourier-Laplace transform}]\label{scalingthmFL}
For all $\kappa, m\in\R, \lambda>0$ we have
\begin{equation}\label{thetaeps}
\lim_{\eps\to 0}\int_0^\infty \mathbb E_{u,w}\[e^{-i(\kappa  (U_\eps(t)-U)+ m W_\eps(t))}\] \, e^{-\lambda t}\, dt=
\left\{
\begin{array}{ll}
\Psi^{R}_W(\kappa,m,\lambda) & \text{in the Reflected Regime,}\\
\\
\Psi^{S}_W(\kappa,m,\lambda,\gamma)& \text{in the Sticky Regime,} \\
\\
\Psi^{A}_W(\kappa,m,\lambda)& \text{in the Absorbed Regime.}
\end{array}
\right.
\end{equation}
\ep

\begin{proof}
Using \col{\eqref{seconda} and} \eqref{platone} we can rewrite the Fourier-Laplace transform  of Theorem \ref{Teo:Green} in integral form:
 \begin{eqnarray}
&& G^{(\theta)}(w,w',\kappa,\lambda)=
\frac 1{2\pi}\int_{-\pi}^\pi \, \frac{f_{\lambda,\kappa}^\theta(w,w')}{2+\lambda-2  \nu_\kappa\cos(\bar m)} \(e^{ i\bar m(w'-w)} + e^{ i\bar m(w'+w)} 
\left( 2\,  \frac{\mathcal Z_{\lambda,\kappa}^{(0)} }{\mathcal Z^{(\theta)}_{\lambda,\kappa}}-1 \right)\)  d\bar m . \nn
\end{eqnarray}
%
%
Furthermore, by taking the Fourier transform with respect to the $w$ variable, one finds
 \begin{eqnarray}
&&\int_0^\infty \mathbb E_{u,w}\[e^{-i\kappa (u(t)-u) -im w(t)}\] \, e^{-\lambda t}\, dt \\
&&= \sum_{w'\ge 0}e^{-im w'}  G^{(\theta)}(w,w',\kappa,\lambda)\nn \\
&&= \frac 1{2\pi}\int_{-\pi}^\pi \, \frac{\sum_{w'\ge 0}f^{(\theta)}_{\lambda,\kappa}(w,w')\, e^{ i(\bar m-m)w'}}{2+\lambda-2  \nu_\kappa\cos(\bar m)} \,
\(e^{-i \bar m w}+e^{i  \bar m w} 
\left( 2\,  \frac{\mathcal Z_{\lambda,\kappa}^{(0)} }{\mathcal Z^{(\theta)}_{\lambda,\kappa}}-1 \right)\) d\bar m. \nn
\end{eqnarray}
Hence, for the full Fourier-Laplace transform of the scaled process, one has
 \begin{eqnarray}
 \label{what3}
&&\int_0^\infty \mathbb E_{u,w}\[e^{-i\sqrt 2(\kappa  (U_\eps(t)-\tfrac{\eps u}{\sqrt 2})+ m W_\eps(t))}\] \, e^{-\lambda t}\, dt = \\
&&= \frac {\eps^2}{2\pi}\int_{-\pi/\eps}^{\pi/\eps} \, \frac{\eps\sum_{w'\ge 0}f_{\eps^2\lambda,\eps\kappa}^{(\theta)}(w,w')\, e^{ i\eps(\bar m- m)w'}}{2+\eps^2\lambda-2  \nu_{\eps\kappa}\cos(\eps \bar m)} \,
\( e^{- iw\eps\bar m} +  e^{ iw\eps\bar m}  
\left( 2\,  \frac{\mathcal Z^{(0)}_{\eps^2\lambda,\eps\kappa}}{\mathcal Z^{(\theta)}_{\eps^2\lambda,\eps\kappa}}-1 \right)\) d\bar m \nn.
\end{eqnarray}
The explicit form of $f_{\lambda,\kappa}^{\theta}(w,w')$ in \eqref{f-funct} gives
\begin{eqnarray}\label{summation}
&&\eps\sum_{w'\ge 0}f_{\eps^2\lambda,\eps\kappa}^{(\theta)}(w,w')\, e^{ i\eps(\bar m- m)w'}= \\
&& \eps\sum_{w'\ge 0} e^{ i\eps(\bar m- m)w'}+ \eps \, \frac \theta 2 \(1+\(\nu_{\eps\kappa}^{-1}\zeta_{\eps^2 \lambda,\eps\kappa}-1\)\mathbf 1_{w=0}\)-\frac \eps 2.\nn\
\end{eqnarray}
Since $\nu_{\eps\kappa}^{-1}\zeta_{\eps^2 \lambda,\eps\kappa}=1+o(\eps)$, 
\col{recalling the definition of the parameter $\gamma$ in \eqref{gamma},}
one has
\begin{eqnarray}
\label{what1}
&&\lim_{\eps \to 0}\eps\sum_{w'\ge 0}f_{\eps^2\lambda,\eps\kappa}^{(\theta)}(w,w')\, e^{ i\eps(\bar m- m)w'}=\int_0^{\infty} e^{ i(\bar m- m)x'}\, dx'+  \, \frac \gamma {\sqrt 2}.
\end{eqnarray}
Similarly one finds
\begin{eqnarray}
\label{what2}
\lim_{\eps \to 0}\,  \frac{\mathcal Z^{(0)}_{\eps^2\lambda,\eps\kappa}}{\mathcal Z^{(\theta)}_{\eps^2\lambda,\eps\kappa}} =
\frac{\sqrt{\kappa^2+\lambda }}{\sqrt{\kappa^2+\lambda }+\gamma\, \sqrt 2 \(\kappa^2+\frac \lambda{2}\)}.
\end{eqnarray}
Hence, taking the limit $\eps\to 0$ in \eqref{what3} and using  \eqref{what1} and \eqref{what2} one has
 \begin{eqnarray}
&&\lim_{\eps\to 0}\int_0^\infty \mathbb E_{u,w}\[e^{-i\sqrt{2}(\kappa  (U_\eps(t)-\tfrac{\eps u}{\sqrt 2})+ m W_\eps(t))}\] \, e^{-\lambda t}\, dt \\
&&= \frac {1}{2\pi}\int_{-\infty}^{\infty} \, \frac{\int_0^{\infty} e^{ i(\bar m- m)x'}\, dx'+  \, \frac \gamma {\sqrt 2} }{\lambda + \kappa^2+\bar m^2} \,
\( e^{- iW\bar m} +  e^{ iW\bar m}  
\left( \frac{\sqrt{\kappa^2+\lambda } - \gamma \, \sqrt 2\(\kappa^2+\frac \lambda{2}\)}{\sqrt{\kappa^2+\lambda }+
 \gamma\, \sqrt 2\(\kappa^2+\frac \lambda{2}\)}
 \right)\)  d\bar m.  \nn
\end{eqnarray}
The previous expression can \col{further simplified using the Fourier transform \eqref{Fourier-formula}.
In the end one arrives to:}
\begin{eqnarray}
&&\lim_{\eps\to 0}\int_0^\infty \mathbb E_{u,w}\[e^{-i\sqrt 2(\kappa  (U_\eps(t)-U)+ m W_\eps(t))}\] \, e^{-\lambda t}\, dt = \\
&&c^{(\gamma)}(\sqrt{2}\kappa,\lambda) \psi^R_{W}(\sqrt 2 \kappa, \sqrt 2 m,\lambda) 
+\left(1- c^{(\gamma)}(\sqrt{2}\kappa,\lambda)\right)\Psi^A_{W}(\sqrt 2 \kappa,\sqrt 2 m,\lambda)\nn
 \end{eqnarray}
%
%
from which   \eqref{thetaeps} follows. 
\end{proof}
\vskip.5cm
\noindent
{\bf  {{\small P}{\scriptsize ROOF OF}} {{\small T}{\scriptsize HEOREM}} \ref{scalingthm}.}
First note that proposition \ref{scalingthmFL} shows the convergence of the Fourier-Laplace transform of the transition probabilities.
So the only left to prove is that we can get rid of the Laplace transform and have convergence in the time parameter, instead of the Laplace parameter.
This is possible because convergence of resolvents implies convergence of semigroups. More precisely,
denote by $T_\eps (t)$ the semigroup of the process $(U_\eps(t)-U, W_\epsi (t))$, and let $T(t)$ be the semigroup of the claimed limiting
process $(U(t),W(t))$, and $A_\epsi, A$ the corresponding generators. By \eqref{thetaeps}, we conclude that for compactly supported
smooth functions $f:\R^2\to\R$ , the resolvents converge, i.e., for all $\lambda>0$ 
\begin{equation*}
\lim_{\eps\to 0} (\lambda - A_\eps)^{-1}f =  (\lambda - A)^{-1}f.
\end{equation*}
Therefore, as smooth functions form a core for all $A_\epsi$ as well as for $A$  by \cite{bob}, Theorem 2.2, we conclude also convergence of the semigroups, i.e., for all
compactly supported continuous functions we have
$$
\lim_{\epsi\to 0} T_\eps (t)f= T(t)f,
$$
which in turn implies the convergence of the processes in the sense of finite dimensional distributions.

\qed


%

\subsection{Local time at 0}
\label{local-sect}

\bl[\col{Laplace transform of probability to be at zero of sticky Brownian motion}]
\label{L:ciao}
Let $z\ge 0$ \col{and let $\mathbb{P}_z(B^S(t) =0)$ be the  probability  for a sticky Brownian
motion started at $z$ to be at 0 at time $t$.  We have}
\be\label{ciao}
\int_0^\infty e^{-\lambda t}\, \col{\mathbb{P}_z(B^S(t) =0)}
\,dt= \frac {\gamma}{\sqrt{2\lambda}+\gamma\lambda} \, e^{-\sqrt{2\lambda}z}.
\ee
\el
\bpr
The l.h.s. of \eqref{ciao} can be rewritten as
\beq
&&\lim_{M\to \infty}\frac 1 {2M}\int_{-M}^M\int_0^\infty e^{-\lambda t}\, \mathbb E_z[e^{-im|B(s(t))|}]\, dt \, dm= \lim_{M\to \infty}\frac 1 {2M}\int_{-M}^M \Psi_z^{S}(0,m,\lambda,\gamma)\, dm \nn \\
\eeq
Thus, using formula \eqref{PsiS}, we have that 
\beq
\int_0^\infty e^{-\lambda t}\, \col{\mathbb{P}_z(B^S(t) =0)}
\,dt & = &  
\col{c^{(\gamma)}(0,\lambda)\;
\; \lim_{M\to \infty}\frac 1 {2M}\int_{-M}^M \Psi^R_{z}(0, m,\lambda)\, dm}  \nn \\
 &+&
\col{(1-c^{(\gamma)}(0,\lambda))
\;\lim_{M\to \infty}\frac 1 {2M}\int_{-M}^M \Psi^A_{z}(0, m,\lambda)\, dm.}\nn
\eeq
It is easy to see that  
\beq
\lim_{M\to \infty}\frac 1 {2M}\int_{-M}^M \Psi^R_{z}(0, m,\lambda)\, dm
=0,
\eeq
and 
\be
\col{(1-c^{(\gamma)}(0,\lambda))
\;\lim_{M\to \infty}\frac 1 {2M}\int_{-M}^M \Psi^A_{z}(0, m,\lambda)\, dm= 
\frac {\gamma}{\sqrt{2\lambda}+\gamma\lambda} \, e^{-\sqrt{2\lambda}z}.}
\ee
\epr
\vskip.3cm
\noindent
{\bf  {{\small P}{\scriptsize ROOF OF}} {{\small P}{\scriptsize ROPOSITION}} \ref{local0}.}
The first statement \col{(equations \eqref{primus} and \eqref{Zeta})} follows from the fact that
\be
\int_0^\infty e^{-\lambda t}\,\col{\mathbb P_{w}\(w(t)=0\)}\, dt= G^{(\theta)}(w,0,0,\lambda)
\ee
\col{and the r.h.s. can be explicitly written thanks to Theorem \ref{Teo:Green}.}
Furthermore, the diffusive scaling gives
\beq
\int_0^\infty e^{-\lambda t}\,\col{\mathbb P_{w}\(W_\eps(t)=0\)}\, dt 
&=&\eps^2 G^{(\theta)}(w,0,0,\lambda \eps^2)\nn \\
&=& \eps^2
 { \zeta_{\lambda \eps^2}^{w}  }\; \frac{1+\sqrt2 \, \gamma \eps^{-1} \zeta_{\lambda  \eps^2}^{\mathbf 1_{w=0}} }{\zeta_{\lambda  \eps^2}^{-1}+(\sqrt 2 \,\gamma  \lambda \eps-1)\zeta_{\lambda \eps^2}}\nn\\
&=& \frac{1+\sqrt 2 \, \gamma \eps^{-1} \(1-{\mathbf 1_{w=0}}\eps\sqrt{\lambda}\)  }{1+\eps\sqrt{\lambda}+(\sqrt 2\, \gamma \lambda \eps-1)
(1-\eps\sqrt{\lambda})} \; \eps^2
 \(1-\eps\sqrt{\lambda}  \)^{\sqrt 2 W \eps^{-1}}\;  \cdot (1+o(1))\nn
\eeq
\col{from which formula \eqref{localstick} follows}.
\qed

\section{Applications in processes with duality.}

\subsection{\col{Time dependent covariances}}
\label{six}

\vskip.2cm
\noindent
\col{In this section we  look at the (time dependent) covariance of $\eta_x(t)$ and $\eta_y(t)$ for
the reference process with generator \eqref{referencep} when initially started from a product measure. 
We recall form section \ref{dual-section} that the reference process encompasses the generalized symmetric
exclusion process (SEP(j)), the symmetric inclusion process (SIP(k)) and the independent
random walk process (IRW).} We will denote by $\nu$ the initial product measure and by
\beq\label{prodex}
\rho(x):=\int \eta_x d\nu \qquad \text{and} \qquad
\chi(x):=\int \eta_x(\eta_x-1) d\nu
\eeq
We further denote
\be\label{rhot}
\rho_t(x)= \sum_{y} p_t(x,y) \rho(y)= \int \E_\eta (\eta_x(t)) d\nu
\ee
We will  denote by $X_t,Y_t$, resp. $\widetilde{X}_t, \widetilde{Y}_t$ the positions of  two dual particles,
resp.  two independent particles, and by   $\E_{x,y}$, the corresponding  expectations
when   particles start from $x,y$.
\col{The following proposition describes time-dependent covariances of particle numbers at time $t>0$ when starting
from an arbitrary initial distribution $\nu$, in terms of two dual particles.}

\bp[Time dependent covariances throught dual particles]
\col{Let $\{\eta(t): t\ge 0\}$ be a self-dual process 
with generator \eqref{referencep}} 
and $\alpha=1$.
Then
\beq\label{finalxi}
\Xi^{(\theta)}(t,x,y;\nu)
&=&\left(1+\theta\delta_{x,y}\right)\bigg\{\col{\E_{x,y} \left[\rho(X_t)\rho(Y_t) - \rho({\widetilde{X}_t})\rho({\widetilde{Y}_t})\right]}
\nonumber\\
&+&
\E_{x,y}\left[\mathbf 1_{X_t=Y_t} \left(\frac 1{1+\theta}\, \chi({X_t})-\rho({X_t})^2\right)\right]\bigg\}
\nonumber\\
&+&
\delta_{x,y}\left(\theta \rho_t(x)^2 +\rho_t(x)\right)
\eeq
where
\begin{equation}\label{table}
\theta=\left\{
\begin{array}{ll}
0 & \text{IRW}\\
+\frac{1}{k} & \text{SIP}(k)\\
-\frac{1}{j} & \text{SEP}(j).
\end{array}
\right.
\end{equation}
\ep

\bpr
To prove the theorem we use duality relations. \col{From section \ref{dual-section}}, duality functions for one and two particles dual configurations are given by:
\begin{equation}
D(\col{\delta_x,\eta})= c_1\eta_x,
\end{equation}
\begin{equation}
D(\col{\delta_x+\delta_y, \eta})=
\left\{
\begin{array}{ll}
c_1^2 \eta_x \eta_y &\quad \text{for} \: x \neq y\\
c_2 \eta_x(\eta_x-1)& \quad \text{for} \: x=y,
\end{array}
\right.
\end{equation}
with
\begin{equation}
c_1:=\left\{
\begin{array}{ll}
1 & \text{IRW}\\
\frac{1}{k} & \text{SIP}(k)\\
\frac{1}{j} & \text{SEP}(j)
\end{array}
\right. \qquad
c_2:=\left\{
\begin{array}{ll}
1 & \text{IRW}\\
\frac{1}{k(k+1)} & \text{SIP}(k)\\
\frac{1}{j(j-1)} & \text{SEP}(j).
\end{array}
\right.
\end{equation}
%
Hence, \col{for all cases,} $\theta+1=c_1^2/c_2$. Then we have
\beq\label{prodex1}
 \rho(x) =\frac{1}{c_1} \int D(\delta_x, \eta) d\nu \qquad \text{and}
\qquad \chi(x) =\frac{1}{c_2} \int D(2\delta_x, \eta) d\nu
\eeq
We denote by $p_t(x,y)$ the transition probability for one dual particle to go from $x$ to $y$ in time $t$.
Moreover
we denote by $p_t(x,y;u,v)$ the transition probability for two dual particles to go from $x,y$ to $u,v$ in time $t$.
We consider two cases: the first being $x\not= y$. Using self-duality, we write
\beq
&&\E_\eta \[\eta_x(t)\eta_y(t)\]-\rho_t(x) \E_\eta\[\eta_y(t)\] -\rho_t(y)\E_\eta\[\eta_x(t)\] +\rho_t(x)\rho_t(y)
\nonumber\\
&= &
\frac{1}{c_1^2}\E_\eta \[D(\delta_x+\delta_y, \eta(t))\] -\rho_t(x) \frac{1}{c_1}\E_\eta \[D(\delta_y, \eta(t))\]  
\nonumber\\
&-&
\rho_t(y)\frac{1}{c_1}\E_\eta \[D(\delta_x, \eta(t))\]  +\rho_t(x)\rho_t(y)
\nonumber\\
&=&
\sum_{u\not= v} p_t(x,y;u,v) \eta_u\eta_v + \frac{c_2}{c_1^2}\, \sum_u p_t(x,y;u,u) \eta_u(\eta_u-1)
\nonumber\\
&-& \rho_t(x) \sum_v p_t(y,v)\eta_v -\rho_t(y) \sum_{u} p_t(x,u)\eta_u +\rho_t(x)\rho_t(y).
\eeq
We now integrate the $\eta$-variable over $\nu$ and obtain
\beq
\Xi^{(\theta)}(t,x,y;\nu)
&=&
\sum_{u,v} p_t(x,y;u,v) \rho(u)\rho(v)
+
\sum_{u} p_t(x,y;u,u) \left(\frac{1}{1+\theta}\;\chi(u)-\rho(u)^2\right)
\nonumber\\
&-&
\rho_t(x)\rho_t(y)-\rho_t(y)\rho_t(x)+ \rho_t(x)\rho_t(y)
\nonumber\\
&=&
\sum_{u,v}\left[p_t(x,y;u,v)-p_t(x,u)p_t(y,v)\right]\rho(u)\rho(v)
\nonumber\\
&+&
\sum_u p_t(x,y;u,u) \left(\frac{1}{1+\theta}\;\chi(u)-\rho(u)^2\right).
\eeq
Now we turn to the second case $x=y$.
We have
\beq
\mathbb E_\eta\[\eta^2_x(t)\]&=&\frac 1 {c_2}\, \mathbb E_\eta \[D(2\delta_x, \eta(t))\]+\frac 1 {c_1}\mathbb E_\eta\[D(\delta_x, \eta(t))\]\nn\\
&=&\frac 1 {c_2}\, \mathbb E_{2\delta_x} \[D(\delta_{x(t)}+\delta_{y(t)},\eta)\]+\frac 1 {c_1}\mathbb E_{\delta_x}\[D(\delta_{x(t)},\eta)\]\nn\\
&=&\frac 1 {c_2}\, \(c_1^2\sum_{u\neq v} p_t(x,x;u,v) \eta_u \eta_v+c_2\sum_u p_t(x,x;u,u)\eta_u(\eta_u-1)\)+\sum_u p_t(x,u)\eta_{u}.\nn
\eeq
Then
\beq
&&\int\mathbb E_\eta\[\eta^2_x(t)\]d\nu=(1+\theta)\sum_{u\neq v} p_t(x,x;u,v) \rho(u) \rho(v)+\sum_u p_t(x,x;u,u)\chi(u)+\sum_u p_t(x,u)\rho(u).\nn\\
\eeq
This leads to
\beq
\Xi^{(\theta)}(t,x,x;\nu)
&=&
(1+\theta) \sum_{u,v}\left[p_t(x,x;u,v)-p_t(x,u)p_t(x,v)\right]\rho(u)\rho(v)
\nonumber\\
&+&
\sum_u p_t(x,x;u,u) \left(\chi(u)-(1+\theta)\rho(u)^2\right)
\nonumber\\
&+&
\theta\rho_t(x)^2 + \rho_t(x).
\eeq
This completes the proof of the Proposition.
\epr

\col{When the initial measure $\nu$ is assumed to be an 
homogeneous product measure then the expression
of the time dependent covariances via dual particles further simplifies.
This is the content of the next proposition.}

\bp[\col{Case of homogeneous $\nu$}]\label{homo}
Suppose that $\nu$ is a homogeneous  product measure
then, for self-dual processes with generator \eqref{referencep} 
and $\alpha=1$ we have
\beq
\int_0^\infty e^{-\lambda t}\, \Xi^{(\theta)}(t,x,y;\nu) \,dt=\left(1+\theta\delta_{x,y}\right)
\left(\frac {\chi} {\theta+1}-\rho^2\right)\;  \frac{\(1+\theta \zeta_{\lambda}^{\mathbf 1_{\col{x=y}}}\) \, \zeta_{\lambda}^{\col{\sqrt{2}|x-y|}}   }{\zeta_{\lambda}^{-1}+(\theta\lambda-1)\zeta_{\lambda}}
+
\frac{\delta_{x,y}}\lambda \,\left(\theta \rho_\nu^2 +\rho_\nu\right)\nn
\eeq
with $\zeta_\lambda$ as in \eqref{Zeta}.
\ep
\bpr
\col{We see from \eqref{finalxi} that if $\nu$ is an homogeneous product measure then $\rho(X_t) = \rho$.
As a consequence we have}
\beq\label{homo2}
\Xi^{(\theta)}(t,x,y;\nu)
=\left(1+\theta\delta_{x,y}\right)
\left(\frac \chi {(1+\theta)}-\rho^2\right)\,\pee_{x,y}\left(X_t=Y_t\right)+
\delta_{x,y}\left(\theta \rho^2 +\rho\right).
\eeq
\col{Taking the Laplace transform and using \eqref{primus} the result follows.}
\epr
\br
Notice that if $\nu$  is \col{an homogeneous product measure} that satisfies the condition
\begin{equation}\label{cond}
\int \eta_0(\eta_0-1)\, d\nu= (1+\theta) \(\int \eta_0\, d\nu\)^2
\end{equation}
then $\Xi^{(\theta)}(t,x,y;\nu)$ is not depending on $t$ and more precisely, $$\Xi^{(\theta)}(t,x,y;\nu)=0  \qquad \text{for} \quad x\neq y \qquad \text{and} \qquad
\Xi^{(\theta)}(t,x,x;\nu)= \chi+\rho - \rho^2.
$$
This corresponds to the case where $\nu$ is a stationary product measure 
for which \col{the covariance is constantly zero} and the variance \col{is equal at all times to the initial value, that is indeed} given by
$$
\text{Var}_\nu(\eta_0)=\int \eta_0^2 \ d\nu -\rho^2=\chi+\rho-\rho^2 .
$$
\er
\subsection{Scaling of variance and covariances in the sticky regime}

\noindent
{\bf  {{\small P}{\scriptsize ROOF OF}} {{\small T}{\scriptsize HEOREM}} \ref{P:variance}.}
From Proposition \eqref{homo} we have 
 \beq\label{homo1}
&&\int_0^\infty e^{-\lambda t}\, \Xi^{(\theta_\eps)}(\eps^{-a}t,\eps^{-1}x,\eps^{-1}y;\nu) \,dt \nn \\
&& =\eps^{a}\int_0^\infty e^{-\lambda \eps^a s}\, \Xi^{(\theta_\eps)}(s,\eps^{-1}x,\eps^{-1}y;\nu) \,ds\nn\\
&&=\eps^a\left(1+\sqrt 2 \gamma \eps^{-1}\delta_{x,y}\right)
\left(\frac {\chi} {\sqrt 2 \gamma \eps^{-1}+1}-\rho^2\right)\;  \frac{\(1+\sqrt 2 \gamma \eps^{-1} \zeta_{\eps^a\lambda}^{\mathbf 1_{\col{x=y}}}\) \, \zeta_{\eps^a\lambda}^{\sqrt 2 |x-y|}   }{\zeta_{\eps^a\lambda}^{-1}+(\sqrt 2 \gamma \eps^{a-1}\lambda-1)\zeta_{\eps^a\lambda}}
\nn\\
&& +
\frac{\delta_{x,y}}\lambda \,\left(\eps^{-1}\sqrt 2 \gamma \rho^2 +\rho\right)\nn
\eeq
Now we use the fact that $\zeta_\delta = (1-\sqrt \delta)(1+o(1))$ for small $\delta$ and we obtain that, for $x\neq y$,
 \beq
\int_0^\infty e^{-\lambda t}\, \Xi^{(\theta_\eps)}(\eps^{-a}t,\eps^{-1}x,\eps^{-1}y;\nu) \,dt =
-\rho^2\sqrt 2 \gamma\;  \frac{ (1-\sqrt{\eps^a\lambda})^{\sqrt 2|x-y|}   }{2\sqrt{\lambda}\,\eps^{-(\tfrac a 2-1)}+\sqrt 2 \gamma \lambda} \, \cdot (1+o(1))\nn
\eeq
as $\eps \to 0$. This produces the result for the covariance.
For $x=y$, we get
\beq
&&\int_0^\infty e^{-\lambda t}\, \Xi^{(\theta_\eps)}(\eps^{-a}t,\eps^{-1}x,\eps^{-1}x,\nu) \,dt =\nn\\
&&= \left\{\left({\chi}-\left(1+\sqrt 2 \gamma \eps^{-1}\right)\rho^2\right)\;  \frac{\(\eps+\sqrt 2 \gamma(1-\sqrt{\eps^a\lambda})\) }{2\sqrt{\lambda}\,\eps^{-(\tfrac a 2-1)}+\sqrt 2 \gamma \lambda}
+
\frac{1}\lambda \,\left(\sqrt 2 \gamma \eps^{-1} \rho^2 +\rho\right)\right\} \cdot (1+o(1))\nn
\eeq
as $\eps \to 0$, from which the statement for the variance follows.
\qed

\subsection{Variance of the density fluctuation field}
{\bf  {{\small P}{\scriptsize ROOF OF}} {{\small T}{\scriptsize HEOREM}} \ref{variancethm}.}
\col{From the definitions of the variance of the density fluctuation field (Eq. \eqref{densfi}) and of the time dependent
covariances of the occupation numbers (Eq. \eqref{xi-def}), we have}
\be\label{fieldvariance}
\E_{\nu}\big[ \left(\mathcal X_\eps(\Phi,\eta,t)\right)^2\big]= \eps^2\sum_{x,y}\Phi(\eps x)\Phi(\eps y)\,\Xi^{(\theta_\eps)}(\eps^{-2}t, x,y;\nu).
\ee
\col{Using Theorem \ref{P:variance} for the time dependent covariances} we get
\beq
&&\int_0^\infty e^{-\lambda t}\;\E_{\nu}\big[ \left(\mathcal X_\eps(\Phi,\eta,t)\right)^2\big]\, dt\nn\\
&&= \col{\eps^2}\sum_{x\neq y}\Phi(\eps x)\Phi(\eps y)\int_0^\infty e^{-\lambda t}\Xi^{(\theta_\eps)}(\eps^{-2}t, x,y;\nu)\, dt+  \col{\eps^2}\sum_{x}\Phi(\eps x)^2\int_0^\infty e^{-\lambda t}\Xi^{(\theta_\eps)}(\eps^{-2}t, x,x;\nu)\, dt\nn\\
&&= \eps^2\sum_{x\neq y}\Phi(\eps x)\Phi(\eps y) \frac{\gamma \rho^2\,e^{-\sqrt{\lambda}\eps |x-y|}}{\sqrt{2\lambda}+\gamma \lambda}\,(1+o(1))+  \eps\sum_{x}\Phi(\eps x)^2\frac{2 \sqrt 2 \gamma\rho^2}{2\lambda+\gamma \lambda\sqrt{2\lambda}}\,(1+o(1))\nn
\eeq
from which the result follows.
\qed

{\section*{Acknowledgements}
FR acknowledges Stefan Grosskinsky for useful discussions in an early stage of this work.
The work of CG is supported in part by the Italian
Research Funding Agency (MIUR) through FIRB project grant n. RBFR10N90W.

\end{document}